\documentclass[12pt]{amsart}

\usepackage{amssymb,amsrefs,amsmath,amsthm,graphicx,enumerate,amsfonts,scalefnt}
\usepackage[all]{xy}

\usepackage{geometry}
\newtheorem{theorem}{Theorem}[section]
\newtheorem{prop}[theorem]{Proposition}
\newtheorem{lem}[theorem]{Lemma}
\newtheorem{claim}[theorem]{Claim}
\newtheorem{eg}[theorem]{Example}

\newtheorem{cor}[theorem]{Corollary}
\newtheorem{defi}[theorem]{Definition}
\newtheorem{rmk}[theorem]{Remark}
\newtheorem{fact}[theorem]{Fact}

\newtheorem*{varthm2}{Theorem \ref{local-bound}}
\newtheorem*{ack}{Acknowledgements}
\newtheorem{conj}[theorem]{Conjecture}

\newcommand{\R}{\mathbb{R}}  
\newcommand{\Q}{\mathbb{Q}}  
\newcommand{\C}{\mathbb{C}}  
\newcommand{\CP}{\mathbb{CP}^{m+1}}  
\newcommand{\Z}{\mathbb{Z}}
\newcommand{\K}{\mathbb{K}_n} 
\newcommand{\Kt}{\mathbb{K}_n[t^{\pm 1}]} 
\newcommand{\kt}{k_n[t^{\pm 1}]} 
\newcommand{\ktp}{k_n[t^{\pm 1}]/{(p(t))}} 
\newcommand{\U}{\mathcal{U}}

\begin{document}

\title{ Higher-order Alexander invariants of hypersurface complements}

\begin{center}
\author{Yun Su}
\end{center}

\begin{abstract}
We define higher-order Alexander modules $\mathcal{A}_{n,i}(\U)$ and higher-order degrees $\delta_{n,i}(\U)$ which are invariants of a complex hypersurface complement $\U$. These invariants come from the module structure of the homology of certain solvable covers of the hypersurface complement. Such invariants were originally developed by T. Cochran in \cite{Co} and S. Harvey in \cite{Ha}, and were used to study knots and 3-manifolds. In this paper, I generalize the result proved by C. Leidy and L. Maxim \cite{HP} from the plane curve complements to higher-dimensional hypersurface complements.  
\end{abstract}

\maketitle

\section{Introduction} 

Zariski observed that the position of singularities on a singular complex plane curve affects the topology of the curve, and the fundamental group of the complement of this curve can detect this fact. However, the fundamental group is in general hard to handle. Therefore, it is natural to seek other tools which can capture the information about the topology, such as the Alexander-type invariants. The notion of Alexander invariants comes from knot theory (See \cite{RB}). Libgober developed such invariants of total linking number infinite cyclic cover in \cite{AP}, \cite{AP2}, \cite{Fun}, invariants of the universal abelian cover are considered by Dimca, Libgober and Maxim in \cite{Mu}, \cite{On}, \cite{Cha}, while invariants of certain solvable covers are studied by Leidy and Maxim in \cite{HP}.

The first Alexander module is defined using the first homology group of the infinite cyclic cover of the curve complement corresponding to the total linking number homomorphism. One of the most important results of Libgober is that the (global) Alexander polynomial of a singular irreducible curve divides the product of the `local' polynomials associated with each singular point (See \cite{AP}, \cite{AP2}). This suggests that the topology of a singular curve is controlled by its singularities.

 In the case when hypersurfaces carry non-isolated singularities and are in general position at infinity, Maxim showed that the infinite cyclic Alexander modules of such hypersurfaces can be realized as intersection homology modules of the ambient space, with a certain local coefficient system (See \cite{MS}). He extended Libgober's divisibility results, and proved not only that the global Alexander polynomials are entirely determined by the local information of the link pairs of singular strata, but also that the zeros of these polynomials are among the roots of unity of order $d$, where $d$ is the degree of the hypersurface.

Similarly, multivariable Alexander invariants of hypersurface complements can be associated to the universal abelian cover of such complements. They enjoy similar finiteness properties by works of Dimca, Ligober, Maxim, Suciu, etc. (See \cite{Mu},  \cite{On}, \cite{Cha}, \cite{Su})

Inspired by the success of using the infinite cyclic cover and universal abelian cover, Leidy and Maxim started to look at the higher-order covers of plane curve complements (in \cite{HP}). The notion of higher-order Alexander invariants was developed originally in knot theory by T. Cochran (in \cite{Co}) and S. Harvey (in \cite{Ha}). These are Alexander-type invariants of coverings corresponding to terms of the derived series of a knot group.  Even though there are difficulties in working with modules over non-commutative rings, Cochran showed that these invariants are very useful for estimating
knot genus, detecting fibered, prime and alternating knots, as well as knot concordance (See \cite{Co}, \cite{C}). 
These invariants have applications if one considers the fundamental group of a link complement or that of a compact, orientable 3-manifold.
For example, Harvey showed that such invariants give lower bounds for the Thurston norm and provide new algebraic obstructions to a 4-manifold of the form $M^3 \times S^1$ admitting a symplectic structure (See \cite{Ha},　 \cite{H}).

Motivated by their application in knot theory and low-dimensional topology, Leidy and Maxim adapted the notion of higher-order Alexander invariants in the study of plane curve complements. To any affine plane curve
$C$, they associate a sequence $\{\delta_n(C)\}_{n\geq 0}$ of integers, called the higher-order degrees of $C$. Roughly speaking, these
integers measure the sizes of quotients of successive terms in the rational derived series of $G = \pi_1(\C^2 \setminus C)$
and are invariants of the fundamental group of the curve complement. Leidy and Maxim's main result asserts that for curves
in general position at infinity, these invariants are finite. This is done by providing an upper bound on the
higher-order degrees of the curve in terms of the corresponding local invariants at the singular points, and a
uniform upper bound depending only on the degree of the curve. This result yields new obstructions on the
type of groups that can arise as fundamental groups of plane curve complements (See \cite{Os}).

  In this paper, I extend Leidy and Maxim's result to higher-dimensional hypersurfaces. Given a reduced hypersurface $V$ in $\mathbb{CP}^{m+1}$ and a fixed generic hyperplane $H$ in $\mathbb{CP}^{m+1}$, we define 
  the affine hypersurface complement $\U:=\mathbb{CP}^{m+1} \setminus (V \cup H)$. Starting with the rational derived series $G_r^{(n)}$ of $G = \pi_1(\U)$, we get a sequence of covers $\U_{\Gamma_n}$ corresponding to the quotients  
 $\Gamma_n:=G/G_r^{(n+1)}$, the so-called higher-order covers. 
For instance, the universal abelian cover is one of the higher-order covers. Furthermore, if the hypersurface is irreducible, then the universal abelian cover is the infinite cyclic cover, and so the usual Alexander modules coincide with the 0-th order Alexander modules $H_i(\U_0;\Z)$. 
We define the higher-order Alexander modules of the hypersurface complement to be $\mathcal{A}_{n,i}^{\Z}(\U)=H_i(\U;\Z\Gamma_n)$, and note that $\mathcal{A}_{0,i}^{\Z}(\U)$ is just the universal abelian Alexander modules of $\U$. In section 2 to 4, we talk about basic definitions and properties about the higher-order Alexander invariants.

In section 5, considering the link at infinity, Maxim showed (in \cite{L2}) that for hypersurface transversal to the hyperplane at infinity,  
the higher-order degrees of the hypersurface complement are bounded above by the higher-order degrees of the link at infinity. Moreover, these invariants are finite.

  In section 6 to 8, we find `local' bounds for the higher-order degrees of the hypersurface complement. In particular, using a Whitney stratification of $V$, we associate to any stratum in $V$ a link pair, which is a topological invariant of the stratum. The main result is as follows:
 
\begin{varthm2}
Let $V$ be a reduced hypersurface in $\CP$ with generic hyperplane $H$ at infinity and let $\U=\CP\setminus V \cup H$. Then for $i\leq m$:

$$\delta_{n,i}(\U) \leq
\sum \theta(k_p, a)\cdot \delta _{n,b}(S^{2m-2{k_p}+1}\setminus K^{2m-2{k_p}-1})$$  
where $(S^{2m-2{k_p}+1}, K^{2m-2{k_p}-1})$ are the local link pairs and $\theta(k_p, a)$ are some coefficients depending on the local topology of the hypersurface.
 
 \end{varthm2}
 
   Since the higher-order degrees of each local link pair are finite, we also get a proof of finiteness on the higher-order degrees of the hypersurface complement. 
  
  In section 9, we introduce how to use fox calculus to calculate first higher-order degrees, and we compute some explicit examples. In section 10, we consider higher-order degrees of a group and try to find some useful properties.

\begin{ack}
The author would like to thank Laurentiu Maxim, Constance Leidy, Tim Cochran, Shelly Harvey, Anatoly
Libgober, Alexandru Dimca for many helpful inspiration and conversations about
this project.
\end{ack}

\section{Rational Derived Series}
In this section, we review the definitions and basic concepts that we will need from \cite{Ha} and \cite{Co}. More details can be found in  these sources.

\begin{defi}
Let $G_r^{(0)}=G$ be a group. For $n\geq 1$, we define the $n^{th}$ term of the rational derived series of $G$ inductively by:
$$G_r^{(n)}= \{  g\in  G_r^{(n-1)} \ | \ g^k \in [G_r^{(n-1)}, G_r^{(n -1)}], \text{ for some } k \in \Z \setminus \{0\} \}.$$
\end{defi}

Note that $G=G_r^{(0)}\supseteq G_r^{(1)} \supseteq \cdots $, denote $\Gamma _n := G/{G_r^{(n+1)}}$ and let $\phi_n:G\twoheadrightarrow G/{G_r^{(n+1)}}$ be the quotient map. Since $G_r^{(n)}$ is a normal subgroup of $G_r^{(i)}$ for $0\leq i \leq n$ (\cite{Ha}, Lemma 3.2), it follows that $\Gamma _n$ is a group.
We use rational derived series as opposed to the usual derived series in order to avoid zero-divisors in the group ring $\Z \Gamma_n$.

It is shown in \cite{Ha}, Lemma 3.5, that the successive quotients of the rational derived series are torsion-free abelian groups. That is for all $n\geq0$,
$$G_r^{(n)}/G_r^{(n+1)}\cong \left(G_r^{(n)}/[G_r^{(n)}, G_r^{(n)}]\right)/\{\Z -\text{torsion}\}.$$

Suppose $X$ is a topological space and $G=\pi_1(X)$, this shows $$G_r^{(n)}/G_r^{(n+1)}\cong H_1(X_{\Gamma_{n-1}};\Z)/\{\Z -\text{torsion}\},$$ where $X_{\Gamma_{n-1}}$ is the regular $\Gamma_{n-1}$ cover of $X$. In particular, $$G/G_{r}^{(1)}=H_1(X)/\{\Z -\text{torsion}\}=\Z ^{\beta _1(X)},$$ where $\beta_1(X)$ is the first betti number of $X$.

\begin{cor}\label{co} (\cite{Ha}, section 3)

1) If $N\lhd G_r^{(n)}$ and $G_r^{(n)}/N$ is torsion-free abelian, then $G_r^{(n+1)}\subseteq N$.

2) If $G^{(n)}/G^{(n+1)}$ is  $\Z$-torsion free for all $n$, then $G_r^{(n)}=G^{(n)}$ for all $n$. 

(Here $G^{(n)}=[G^{(n-1)},G^{(n-1)}]$ is the usual derived series.)
\end{cor}

\begin{eg}\rm
1) If $G$ is a finite group, then $G=G_r^{(n)}$ and $\Gamma _n=1$ for all $n$.

2) If $G$ is a knot group or a free group, then the quotients of successive terms of the derived series are torsion free abelian. Hence $G_r^{(n)}=G^{(n)}$ for all $n$. The rational derived series and the derived series coincide (See \cite{Ha}, p. 902).

3) If $G$ is the fundamental group of a link complement in $S^3$ or that of a plane curve complement, then $G_r^{(1)}=G^{(1)}$.
\end{eg}

\begin{defi} 
A group $\Gamma$ is poly-torsion-free-abelian (PTFA) if it admits a normal series $1=G_0 \lhd G_1 \lhd ...\lhd G_n=\Gamma$ such that each of the factors $G_{i+1}/G_i$ is torsion free abelian.

\end{defi}

\begin{prop} We collect the following facts from \cite{Ha}:

1) If $A\subseteq G$ and $G$ is PTFA, then $A$ is PTFA.

2) If $A\lhd G$ is PTFA and $G/A$ is PTFA, then $G$ is PTFA.

3) Any PTFA group is torsion free and solvable.

4) 
$\Gamma_n$ is a PTFA group.
\end{prop}

\section{Skew Laurent Polynomial Ring}
The following definitions and properties are shown in \cite{Ha}.

\begin{defi}
Let $R$ be a ring and $S$ be a subset of $R$. $S$ is a right divisor set of $R$ if the following properties hold:

1) $0\notin S$, $1\in S$,

2) $S$ is multiplicatively closed,

3) Given $r\in R$, $s\in S$, there are $r_1\in R$, $s_1\in S$ such that $rs_1=sr_1$.

\end{defi}

\begin{defi}
The right quotient ring $RS^{-1}$ is a ring containing $R$ such that:

1) every element of $S$ has an inverse in $RS^{-1}$,

2) every element of $RS^{-1}$ is of the form $rs^{-1}$ with $r\in R$, $s\in S$.

\end{defi}

It is known that if $S\subseteq R$ is a right divisor set, then the right quotient ring $RS^{-1}$ exists.

\begin{defi}
If $R$ has no zero-divisors and $S=R\setminus \{0\}$ is a right divisor set, then $R$ is called an Ore domain.
If $R$ is an Ore domain, then $RS^{-1}$ is a skew field, called the classical right ring of quotients of $R$.

\end{defi}

\begin{prop} (\cite{Ha}, prop. 4.1, remark 4.2, remark 4.3)

1) If $\Gamma$ is a PTFA, then $\Z\Gamma$ is a right Ore domain, i.e. $\Z\Gamma$ embeds in its classical right ring of quotients $\mathcal{K}=\Z\Gamma(\Z\Gamma \setminus \{0\})^{-1}$ which is a skew field.

2) If $R$ is an Ore domain and $S$ is a right divisor set, then $RS^{-1}$ is flat as a left $R$-module. In particular, $\mathcal{K}$ is flat left $\Z\Gamma$-module.

3) Every module over $\mathcal{K}$ is a free module, and such module has a well-defined rank $rk_{\mathcal{K}}$, which is additive on short exact sequences.

\end{prop}

\begin{defi}
If $M$ is a right R-module with R an Ore domain, then the rank of M is defined as 
rank(M)$=rk_{\mathcal{K}}(M\otimes_R\mathcal{K})$, where $\mathcal{K}=R(R\setminus \{0\})^{-1}$.

\end{defi}

In particular, $A$ is a torsion $R$-module if and only if $A\otimes_R\mathcal{K}=0$.

\begin{rmk}\rm (\cite{Ha}, remark 4.4)

If $\mathcal{C}=\{(C_i,\partial _i)\}_{i\geq 0}$ is a finite chain complex of finitely generated free right $\Z\Gamma$
modules, then the Euler characteristic $\chi(\mathcal{C})=\sum\limits_{i=0}^{\infty}
(-1)^i \text{rank}(C_i)$ is defined and is equal to $\sum\limits_{i=0}^{\infty}(-1)^i  \text{rank}(H_i(\mathcal{C}))$.

\end{rmk}

The rest of this section is devoted to the ring $\Kt$, which are obtained from $\Z \Gamma_n$ by inverting the non-zero elements of a particular subring described below. This construction is used in \cite{Ha} and \cite{Co}.

Consider the group $\Gamma_n=G/G_r^{(n+1)}$ for $n\geq 0$. Since $\Gamma_n$ is PTFA, $\Z\Gamma_n$ embeds in its right ring of quotients $\mathcal{K}_n=\Z\Gamma_n(\Z\Gamma_n \setminus \{0\})^{-1}$.
Let $\psi\in H^1(G;\Z)$ be primitive, i.e. $\psi:G\twoheadrightarrow\Z$ is an epimorphism. Notice that $G/{\text{ker} \psi}\cong\Z$ is torsion free abelian; by corollary 2.2, $G_r^{(n+1)}\subseteq ...\subseteq G_r^{(1)} \subseteq \text{ker} \psi$. So $\psi$ can be factored as:
$$\xymatrix{
              G \ar[r]_{\psi}  \ar[d]_{p}          &  \Z \\
 G/G_r^{(n+1)}=\Gamma_n   \ar[ru]_{\bar{\psi}}      &  \\                              
}$$

Let $\bar{\Gamma}_n$ be the kernel of $\bar{\psi}$. Since $\bar{\Gamma}_n$ is a subgroup of $\Gamma_n$, $\bar{\Gamma}_n$ is PTFA by proposition 2.5. Thus $\Z\bar{\Gamma}_n$ is an Ore domain and $\Z\bar{\Gamma}_n \setminus \{0\}$ is a right divisor set of $\Z\bar{\Gamma}_n$. Let\linebreak
 $\mathbb{K}_n=\Z\bar{\Gamma}_n(\Z\bar{\Gamma}_n \setminus \{0\})^{-1}$ be the right ring of quotients of $\Z\bar{\Gamma}_n$, and set $R_n=\Z \Gamma_n(\Z\bar{\Gamma}_n \setminus \{0\})^{-1}$ which is a non-commutative principal left and right ideal domain (i.e. it has no zero divisor and every left and right ideal is principal).

Since $\Gamma_n/\bar{\Gamma}_n \cong \Z$, we can choose a splitting $\xi:\Z\rightarrow\Gamma_n$ defined by $1\rightarrow t$. 
As in proposition 4.5 of \cite{Ha}, the embedding $i:\Z\bar{\Gamma}_n\hookrightarrow \mathbb{K}_n$ extends to an isomorphism $R_n \cong \mathbb{K}_n[t^{\pm 1}]$. However, this isomorphism depends on the choice of the splitting.

Now we have $$\Z\Gamma_n\hookrightarrow R_n =\Z\Gamma_n(\Z\bar{\Gamma}_n \setminus \{0\})^{-1}\cong \mathbb{K}_n[t^{\pm 1}]\hookrightarrow \mathcal{K}_n=\Z\Gamma_n(\Z\Gamma_n \setminus \{0\})^{-1}.$$ Notice that $\Z\Gamma_n$ and $\mathcal{K}_n$ only depend on G, but $\mathbb{K}_n[t^{\pm 1}]$ depends on $\psi :G\twoheadrightarrow \Z$ and the splitting $\xi :\Z\rightarrow \Gamma_n$. Moreover, $\mathbb{K}_n[t^{\pm 1}]$ and $\mathcal{K}_n$ are flat $ \Z\Gamma_n $-modules by prop 3.4.

\section{Definition of Higher-Order Alexander Invariants}
Let $X$ be a connected CW complex and $\phi :\pi_1(X)\rightarrow \Gamma $ be a homomorphism. Let $X_{\Gamma}$ be the corresponding $\Gamma$-cover. Then the singular chains on $X_{\Gamma}$, $C_*(X_{\Gamma})$, gets an induced right $\Z \Gamma$-module structure, via the deck group action. A similar structure can be obtained for a pair $(X, A)$.

\begin{defi}
Let M be a $\Z\Gamma$-bimodule, define

$H_*(X;M)=H_*(C_*(X_{\Gamma};\Z)\otimes_{\Z\Gamma}M)$ and 

$H_*(X,A;M)=H_*(C_*(X_{\Gamma},A_{\Gamma};\Z)\otimes_{\Z\Gamma}M)$
as right $\Z\Gamma$-modules,

$H^*(X;M)=H_*(\text{Hom}_{\Z\Gamma}(C_*(X_{\Gamma};\Z),M))$ and

$H^*(X,A;M)=H_*(\text{Hom}_{\Z\Gamma}(C_*(X_{\Gamma},A_{\Gamma};\Z),M))$
as left $\Z\Gamma$-modules.
\end{defi}

\begin{rmk}\rm (\cite{Co}, remark 3.6)

1) By definition, $H_*(X;\Z\Gamma) \cong H_*(X_{\Gamma};\Z)$ as a right $\Z\Gamma$-module. Moreover if $M$ is flat as a $\Z\Gamma$-module, then $H_*(X;M)\cong H_*(X_{\Gamma};\Z)\otimes_{\Z\Gamma}M$. If $\Gamma$ is a PTFA, then $H_*(X_{\Gamma};\Z)\cong H_*(X;\Z\Gamma)$ is a torsion module if and only if $H_*(X;\mathcal{K})\cong H_*(X_{\Gamma};\Z)\otimes_{\Z\Gamma}\mathcal{K}=0$.

2) If $X$ is a compact, oriented n-manifold, then by Poincar\'{e} duality $H_p(X;M)$ is isomorphic to $H^{n-p}(X,\partial X;M)$ which can be regarded as a $\Z\Gamma$-module using the obvious involution on this group ring.

3) The universal coefficient spectral sequence in this setting collapses to the universal coefficient theorem for coefficient in a (noncommutative) principal ideal domain (in particular for the skew field $\mathcal{K}$). 
Hence, $H^n(X;\mathcal{K})\cong \text{Hom}_{\mathcal{K}}(H_n(X;\mathcal{K}),\mathcal{K})$.

\end{rmk}

\begin{prop}\label{a2} (\cite{Co}, proposition 3.7)

Let $X$ be a connected CW complex and $\Gamma$ a PTFA group, if $\phi :\pi_1(X)\rightarrow \Gamma $ is nontrivial, then $H_0(X;\mathcal{K})=0$ and $H_0(X;\Z\Gamma)$ is a torsion module.

\end{prop}

\begin{prop}\label{a1} (\cite{Co}, proposition 3.10)

Suppose $\pi_1(X)$ is finitely generated and $\phi :\pi_1(X)\rightarrow \Gamma $ is nontrivial, then 
$$rk_{\mathcal{K}}H_1(X;\mathcal{K})=rk_{\Z\Gamma}H_1(X;\Z\Gamma)\leq \beta_1(X)-1.$$ So if $\beta_1(X)=1$, then $H_1(X;\Z\Gamma)$ is a torsion module.
\end{prop}

Now let $G=\pi_1(X,x_0)$ and $\Gamma_n:=G/G_r^{(n+1)}$ and $\phi_n:G\twoheadrightarrow \Gamma_n$ be the projection. Define the $n$-th order cover $X_{\Gamma_n}\xrightarrow{p_n} X$ to be the regular $\Gamma_n$-cover. Then $\pi_1(X_{\Gamma_n})=G_r^{(n+1)}$ and the deck group is isomorphic to $\Gamma_n$.

If $R$ is any ring with $\Z\Gamma_n\subseteq R\subseteq \mathcal{K}_n=\Z\Gamma_n(\Z\Gamma_n \setminus \{0\})^{-1}$, then $R$ is a $\Z\Gamma_n$-bimodule. Then $H_*(X;R)$ can be considered as a right $R$-module.

\begin{defi} The n-th order Alexander modules of a CW complex $X$ are
$$\mathcal{A}_n^{\Z}(X) = H_1(X; \Z\Gamma_n) = H_1(X_{\Gamma_n}; \Z)=\frac{G_r^{(n+1)}}{[G_r^{(n+1)}, G_r^{(n+1)}]}$$ as a right $\Z \Gamma_n$-module.

Let $\bar{\mathcal{A}}_n^{\Z}(X) = T_{\Z \Gamma_n}H_1(X; \Z\Gamma_n)$, which is the $\Z \Gamma_n$ torsion submodule of $H_1(X; \Z\Gamma_n)$.

\end{defi}

One can define a right $\Z \Gamma_n=\Z[G/G_r^{(n+1)}]$-module structure on $G_r^{(n+1)}/[G_r^{(n+1)}, G_r^{(n+1)}]$ .

\begin{rmk}\rm (See \cite{Ha})

1) $\mathcal{A}_n^{\Z}(X)/\{\Z \text{-torsion}\}\cong G_r^{(n+1)}/G_r^{(n+2)}$.

2) $\mathcal{A}_n^{\Z}(X) $ and $\overline{\mathcal{A}}_n^{\Z}(X)$ only depend on $G=\pi_1(X,x_0)$.
\end{rmk}

\begin{defi}

The $n$-th order rank of $X$ is:
$$r_n(X) = rk_{\mathcal{K}_n}H_1(X;\mathcal{K}_n)=rk_{\Z \Gamma_n}H_1(X;\Z \Gamma_n)$$
\end{defi}

\begin{prop} (\cite{Ha}, propostition 5.6)
Let $\Gamma$ be PTFA and $\phi :\pi_1(X,x_0)\rightarrow \Gamma $ be nontrivial, then $rk_{\mathcal{K}}H_1(X;\mathcal{K})=rk_{\mathcal{K}}H_1(X,x_0;\mathcal{K})-1$ and $T_{R}H_1(X;\mathcal{R})\cong T_{R}H_1(X,x_0;\mathcal{R})$ for any ring $\mathcal{R}$ such that $\Z\Gamma \subseteq \mathcal{R} \subseteq \mathcal{K}$.

\end{prop}

\begin{defi}
Let $X$ be a finite CW complex. For each primitive $\psi \in H^1(X;\Z)$, the $n$-th order localized Alexander module is defined to be 
 $$\mathcal{A}_n(X) = H_1(X;R_n),$$ viewed as a right $R_n$-module, where $R_n=\Z \Gamma_n(\Z\bar{\Gamma}_n \setminus \{0\})^{-1}$. If we choose a splitting $\xi$ to identify $R_n$ with $\mathbb{K}_n[t^{\pm 1}]$, we define $$\mathcal{A}_n^{\xi}(X) = H_1(X;\mathbb{K}_n[t^{\pm 1}])$$ as a right  $\mathbb{K}_n[t^{\pm 1}]$-module. Set
$$\bar{\mathcal{A}}_n(X) = T_{R_n}H_1(X;R_n),$$ viewed as a right $R_n$-module. If we choose a splitting $\xi$ to identify $R_n$ with $\mathbb{K}_n[t^{\pm 1}]$, we define $$\bar{\mathcal{A}}_n^{\xi}(X) = T_{\mathbb{K}_n[t^{\pm 1}]}H_1(X;\mathbb{K}_n[t^{\pm 1}])$$ as a right  $\mathbb{K}_n[t^{\pm 1}]$-module.

\end{defi}

\begin{rmk}\rm (\cite{HP}, remark 3.8)

Notice that $\mathcal{A}_n(X)$ and $\bar{\mathcal{A}}_n(X)$ depend on $\psi \in H^1(X;\Z)$, whereas $\mathcal{A}_n^{\xi}(X)$ and $\bar{\mathcal{A}}_n^{\xi}(X)$ depend on $\psi \in H^1(X;\Z)$ and the splitting $\xi:\Z\rightarrow \Gamma_n$.
\end{rmk}

\begin{rmk}\rm (See \cite{Ha})

Since $\mathbb{K}_n[t^{\pm 1}]$ is a (non-commutative) PID, we have

$$\mathcal{A}_n^{\xi}(X) \cong \left( \oplus_{i=1}^m \frac{\mathbb{K}_n[t^{\pm 1}]}{p_i(t)\mathbb{K}_n[t^{\pm 1}]}\right)\oplus \mathbb{K}_n[t^{\pm 1}]^{r_n(X)}$$

and the torsion part $$\bar{\mathcal{A}}_n^{\phi}(X)\cong \left( \oplus_{i=1}^m \frac{\mathbb{K}_n[t^{\pm 1}]}{p_i(t)\mathbb{K}_n[t^{\pm 1}]}\right)$$ is a finitely generated $\mathbb{K}_n[t^{\pm 1}]$-module.
\end{rmk}

\begin{defi}
The $n$-th order degree of X is defined to be:
$$\delta_n(X) = rk_{\mathbb{K}_n}\mathcal{A}_n(X) = rk_{\mathbb{K}_n}\mathcal{A}_n^{\xi}(X)$$
and 
$$\bar{\delta}_n(X) = rk_{\mathbb{K}_n}\bar{\mathcal{A}}_n(X) = rk_{\mathbb{K}_n}\bar{\mathcal{A}}_n^{\xi}(X)$$
\end{defi}

\begin{rmk}\rm (\cite{Ha}, \cite{HP})

1) If $\delta_n(X)<\infty$, then $\delta_n(X)=\bar{\delta}_n(X)$.

2) $\bar{\delta}_n(X)$ equals to the degree of the polynomial $\prod\limits_{i=1}^m p_i(t)$.
\end{rmk}

\begin{rmk} \rm
The followings are equivalent: (\cite{HP})

1) $r_n(X)=rk_{\mathcal{K}_n}H_1(X;\mathcal{K}_n)=rk_{\Z\Gamma_n}H_1(X;\Z\Gamma_n)=0 $,

2) $\mathcal{A}_n^{\xi}(X)=\bar{\mathcal{A}}_n^{\xi}(X)$ as torsion $\mathbb{K}_n[t^{\pm 1}]$-modules for any $\xi$,

3) $\delta_n(X)<\infty$,

4) $\mathcal{A}_n^{\Z}(X)$ is a torsion $\Z \Gamma_n$-module,

5) $\mathcal{A}_n(X)=\bar{\mathcal{A}}_n(X)$ as a torsion $R_n$-module.
\end{rmk}

\section{Higher-Order Invariants of Complex Hypersurface Complement}

In this section, we define the higher-order Alexander modules of a complex hypersurface in general position at infinity. For the case of plane curves, see \cite{HP}.

Let $V$ be a reduced hypersurface in $\CP (m\geq 2)$ of degree $d$, defined by a homogeneous equation $f=f_1\cdots f_r=0$ with $r$ irreducible components $V_i=\{f_i=0\}$. We fix a generic hyperplane $H$ in $\CP$, which is called the hyperplane at infinity. Let $V_a=V \setminus (V\cap H)$ be the affine part of $V$. Since $H$ is generic, $V$ and $H$ interest transversely in the stratified sense. Let $\U$ be the affine hypersurface complement $$\U:=\CP \setminus (V\cup H) = \C^{m+1} \setminus V_a$$ with the natural identification of $\C^{m+1}$ and $\CP \setminus H$.

It is known that $H_1(\U) \cong \Z^r$, generated by the meridian loops $\gamma_i$ about the nonsingular part of each irreducible component $V_i$, for $i=1,...,r$. If $\gamma_{\infty}$ denotes the meridian about the hyperplane at infinity, then there is a relation: $ \gamma_{\infty}+\sum d_i\gamma_i=0$, where $d_i=\text{deg}(V_i)$ (cf. \cite{DB}, theorem 4.1.3, 4.1.4).

Consider $\psi: \pi_1 (\U) \xrightarrow{ab} H_1(\U) \rightarrow\Z$, which equals to the linking number homomorphism
$$ \pi_1 (\U) \xrightarrow{[\alpha] \rightarrow lk(\alpha,V \cup -dH)} \Z.$$

Note that $\phi:\pi_1(U)\rightarrow \Gamma_n=G/G_r^{(n+1)} $ is surjective. Let $\U_{\Gamma_n}$ be the corresponding $\Gamma_n$-cover.
Under the action of the deck group $\Gamma_n$, $C_*(\U_{\Gamma_n})$ becomes a complex of right $\Z \Gamma_n$-modules.

\begin{defi} The n-th order Alexander modules of $V$ are

$$\mathcal{A}_{n,i}^{\Z}(\U) = H_i(\U; \Z\Gamma_n) = H_i(\U_{\Gamma_n}; \Z)$$ as right $\Z \Gamma_n$-modules.

The n-th order ranks of (the complement of) $V$ are
$$r_{n,i}(\U) = rk_{\mathcal{K}_n}H_i(\U;\mathcal{K}_n)=rk_{\Z \Gamma_n}H_i(\U;\Z \Gamma_n).$$
\end{defi}

\begin{defi}
(1) The n-th order localized Alexander modules of the hypersurface $V$ are defined to be $$\mathcal{A}_{n,i}(\U) = H_i(\U;R_n)$$ as a right $R_n$-module. If we choose a splitting $\xi$ to identify $R_n$ with $\mathbb{K}_n[t^{\pm 1}]$, we define $$\mathcal{A}_{n,i}^{\xi}(\U) = H_i(\U;\mathbb{K}_n[t^{\pm 1}]).$$

(2) The n-th order degrees of V are defined to be:
$$\delta_{n,i}(\U) = rk_{\mathbb{K}_n}\mathcal{A}_{n,i}(\U) = rk_{\mathbb{K}_n}\mathcal{A}_{n,i}^{\xi}(\U).$$
\end{defi}

Since the isomorphism between $R_n$ and $\mathbb{K}_n[t^{\pm 1}]$ depends on the choice of splitting $\xi$, we cannot define a higher-order Alexander polynomial in a meaningful way (such a polynomial depends on the splitting), as we do in the infinite cyclic case. However, the degree of the associated higher-order Alexander polynomial does not depend on the choice of splitting.

Because $\U$ is the complement of an affine hypersurface, $\U$ is isomorphic to a smooth affine hypersurface in 
$\C^{m+2}$, hence is a Stein manifold of complex dimension $m+1$. Therefore $\U$ has the homotopy type of of a finite CW-complex of real dimension $m+1$ (cf. \cite{DB}, theorem 1.6.7, 1.6.8). Hence, we have $H_i(\U;\Kt) =0$ for $i>m+1$ and $H_{m+1}(\U;\Kt)$ is a free $\Kt$-module. 

\subsection{Connection with $L^2$-Betti numbers}

For those of who are interested in $L^2$-Betti numbers, there is a connection between higher-order degrees and $L^2$-Betti numbers (See \cite{L2}).

Let $M$ be a topological space and let $\alpha:\pi_1(M)\rightarrow \Gamma$ be an epimorphism to a group. Denote by $M_{\Gamma}$ the regular covering of $M$ corresponding to $\alpha$, and consider the $\mathcal{N}(\Gamma)$-chain complex
$$C_*(M_{\Gamma})\otimes_{\Z\Gamma}\mathcal{N}(\Gamma)$$  where $\mathcal{N}(\Gamma)$ is the von Neumann algebra of $\Gamma$ and $C_*(M_{\Gamma})$ is the singular chain complex of $M_{\Gamma}$ with right $\Gamma$-action given by covering translation, and $\Gamma$ acts canonically on $\mathcal{N}(\Gamma)$ on the left. 

\begin{defi}

The $i$-th $L^2$-Betti number of the pair $(M,\alpha)$ is defined as
$$ b_i^{(2)}(M,\alpha):=\text{dim}_{\mathcal{N}(\Gamma)}(H_i(C_*(M_{\Gamma})\otimes_{\Z\Gamma}\mathcal{N}(\Gamma)))\in [0,\infty].$$

\end{defi}

Back to our case, $\psi:\pi_1(\U)\rightarrow \Z$ is the total linking number homomorphism and \linebreak
$\phi:\pi_1(\U)\rightarrow \Gamma_n$ is the quotient map. Denote by $\widetilde{\U}$ the infinite cyclic cover of $\U$ and $\widetilde{\phi}:\pi_1(\widetilde{\U})\rightarrow \bar{\Gamma}_n$.

$$\xymatrix{
              \pi_1(\U) \ar[r]_{\psi}  \ar[d]_{\phi}          &  \Z \\
 \Gamma_n   \ar[ru]_{\bar{\psi}}      &  \\                              
}$$
Then we have the followings:
$$r_{n,i}(\U) = \text{rk}_{\mathcal{K}_n}H_i(\U;\mathcal{K}_n)=b_i^{(2)}(\U,\phi)$$
$$\delta_{n,i}(\U) = \text{rk}_{\mathbb{K}_n}H_i(\U;\mathbb{K}_n[t^{\pm 1}])= \text{rk}_{\mathbb{K}_n}H_i(\widetilde{\U};\mathbb{K}_n)=b_i^{(2)}(\widetilde{\U},\widetilde{\phi}).$$
Notice that $\mathcal{K}_n$ is the skew field of $\Z\Gamma_n$ and $\mathbb{K}_n$ is the skew field of $\Z\bar{\Gamma}_n$.

\subsection{Bounds for higher-order degrees}

Applying the results by Maxim in \cite{L2}, we get the following theorem for higher-order degrees of hypersurface complements.
\begin{theorem}\label{infinity-bound} (\cite{L2}, Theorem 3.1, 3.2, 3.3) 
Let $V$ be a reduced hypersurface in $\CP$ transversal to the hyperplane at infinity $H$, and let $S_{\infty}^{2m+1}$ be a sphere in $\C ^{m+1}\cong \CP\setminus H$ of a sufficiently large radius (that is, the boundary of a small tubular neighborhood in $\CP$ of $H$). Denote by $S_{\infty}^{2m+1}\cap V$ the link at infinity, and let $$\U_{\infty}:=S_{\infty}^{2m+1}\setminus S_{\infty}^{2m+1}\cap V.$$ Then the higher-order degrees satisfy $$\delta_{n,i}(\U) \leq \delta_{n,i}(\U_{\infty})$$ for $i\leq m$.  Furthermore, the latter are finite.
\end{theorem}

\begin{cor}\label{finite}
Suppose $V$ is a reduced hypersurface in $\CP$ transversal to the hyperplane at infinity $H$. Let $\U=\CP \setminus (V\cup H)$. Then for $ i\leq m$, the higher-order Alexander modules $H_i(\U;\Z \Gamma_n)$ are torsion $\Z \Gamma_n$-modules.
\end{cor}

Since $\delta_{n,i}(\U)$ are finite, $r_{n,i}(\U)=0$ and $H_i(\U;\Z \Gamma_n)$ are torsion $\Z \Gamma_n$-modules for $i\leq m$.

\begin{cor} (\cite{L2}, Theorem 3.3 (2))
$$H_i(\U;\mathcal{K}_n) =0\text{ for }i\neq m+1$$ and $$H_{m+1}(\U;\mathcal{K}_n)\cong \mathcal{K}_n^{(-1)^{m+1}\chi(\U)}$$ as a $\mathcal{K}_n$-module. 

\end{cor}

\begin{rmk}\rm (See \cite{HP}, remark 3.3, remark 3.4)

If G is the fundamental group of a link complement in $S^3$ or a hypersurface complement, then

1) $U_{\Gamma_0}$ is the universal abelian cover.

2) If $\mathcal{A}_{0,1}^{\Z}(\U)=0$ (this is the case when G is abelian or finite), then $\mathcal{A}_{n,1}^{\Z}(\U)=0$ for all $n \geq 1$. 

\end{rmk}

Here is a table showing the relations between various covers of $\U$.

\begin{scriptsize}

\begin{table}[ht]
\caption{higher-order covers}
\centering
\begin{tabular}{c c c c c}

\hline
 & covers & maps & $\pi_1$ & Deck group \\ [0.5ex]
\hline
universal cover & $\U_0$ & $G\rightarrow G$ & 0 & $G$ \\
 & \vdots &\vdots &\vdots &\vdots \\
nth order cover & $\U_{\Gamma_n}$ & $G \rightarrow \Gamma_n$ & $G_r^{(n+1)}$ & $ \Gamma_n$ \\
 & $\U_{\Gamma_{n-1}}$ & $G\rightarrow \Gamma_{n-1}$ & $G_r^{(n)}$ & $ \Gamma_{n-1}$ \\
 & \vdots &\vdots &\vdots &\vdots \\
universal abelian cover & $\U_{\Gamma_0}$ & $G\rightarrow \Gamma_0$ & $G'$ & $ H_1(\U)=\Z ^r$ \\

infinite cyclic cover & $\widetilde{\U}$ & $G\xrightarrow{\psi} \Z$ & $\text{ker}(\psi)$ & $ \Z$ \\
 & $\U$ & $G\rightarrow 0 $ & $G$ & $ 0$ \\

 \hline

\end{tabular}

\end{table}

\end{scriptsize}

From the table, note that $$\mathcal{A}_{n,1}^{\Z}(\U)=H_1(\U_{\Gamma_n};\Z)=H_1(\U;\Z \Gamma_n)=\frac{G_r^{(n+1)}}{[G_r^{(n+1)},G_r^{(n+1)}]},$$ so $\mathcal{A}_{n,1}^{\Z}(\U)$ depends on the fundamental group $G$ only, whereas the higher-dimensional $\mathcal{A}_{n,i}^{\Z}(\U)$ also depends on the cell-structure of the cover $\U_{\Gamma_n}$.

\begin{rmk}\rm (\cite{HP}, remark 3.9)
If $V$ is irreducible, then $\delta _{0,i}(\U)$ is the degree of the $i$-th Alexander polynomial of $V$.

\end{rmk}

From the table, note that if $V$ is irreducible, $\U_{\Gamma_0}=\widetilde{\U}$ is the infinite cyclic cover of $\U$. Since $\bar{\Gamma}_0\cong 0$ and $\mathbb{K}_0 \cong \Q$, $\delta _{0,i}(\U)=\text{rk} H_i(\U;\mathbb{K}_0[t^{\pm 1}])=\text{rk} H_i(\U;\Q[t^{\pm 1}])$.

\begin{prop}(Generalization of [22], proposition 5.1)
Let $V$ be an irreducible hypersurface in $\C ^{m+1}$. Let $\U=\C ^{m+1}-V$ and $\widetilde{\U}$ be the infinite cyclic cover of $\U$. Let $G=\pi_1(\U)$ and denote $\Delta_i(t)=\text{order}H_i(\widetilde{\U};\Q)$ the $i$-th Alexander polynomial of the hypersurface complement.
If $\Delta_1(t)=1$, then \begin{align*}
\delta_{n,1}(\U)&=0\text{ for all }n,\\
\delta_{n,i}(\U)&=\text{deg}\Delta_i(t)\text{ for all }n\text{ and all }i.
\end{align*}

Moreover, $\mathcal{A}_{n,i}^{\Z}(\U)\cong \mathcal{A}_{0,i}^{\Z}(\U)$ as $\Z[G/G']$-module for all $n$.

\end{prop}

\begin{proof}
Since $V$ is irreducible, $H_1(\U)=\Z$. So $\bar{\Gamma}_0=0$ and $\mathbb{K}_0=\Z(\Z-0)^{-1}=\Q$.
$$\delta_{0,i}(\U)=\text{rank}_{\mathbb{K}_0}H_i(\U;R_0)=\text{rank}_{\mathbb{K}_0}H_i(\widetilde{\U};\mathbb{K}_0)=\text{rank}_{\Q}H_i(\widetilde{\U};\Q)=\text{deg}\Delta_i(t).$$
Now since $\Delta_1(t)=1$, $H_1(\widetilde{\U};\Q)=G'/G''\otimes\Q=0$. i.e. $G'/G''$ is a torsion abelian group. So
$$G_r'/G_r'' \cong(G_r'/[G_r',G_r'])/\{\Z \text{-torsion} \} \cong (G'/G'')/\{\Z \text{-torsion} \} \cong 0.$$
Hence $G_r''\cong G_r'=G'$, by induction, $G_r^{(n+1)}=G'$ for all $n\geq 0$. So $\Gamma_n =G/G_r^{(n+1)}=G/G'=\Gamma_0$ and $\mathbb{K}_n=\mathbb{K}_0$ for all $n$. 
Therefore, $\delta_{n,i}(\U)=\delta_{0,i}(\U)=\text{deg}\Delta_i(t)$ for all $n$ and all $i$. In particular,  $\delta_{n,1}(\U)=\text{deg}\Delta_1(t)=0$.

\end{proof}

Assume that $s$ is the dimension of $\text{Sing} V$, the singularities of $V$, so $0 \leq s\leq m-1$. If $s < m-1$, then $V$ is irreducible. Furthermore, Libgober showed that:

\begin{lem} (\cite{HG}, lemma 1.5, \cite{Li1})

If $s<m-1$, $\pi _1(\U)=\Z$ and $\pi_i(\U)=0$ for $2\leq i<m-s$. 
\end{lem}

If $m\geq 2$ and $s<m-1$, then $\pi_1(\U)=\Z$, so all of the covers of $\U$ coincide. Thus all the higher-order Alexander modules are just the usual Alexander modules. In this case, $\mathcal{A}_{n,i}^{\Z}(\U)=H_i(\U_0;\Z)$ and the Alexander polynomial $\Delta_1(t)=1$, so the higher-order degrees $\delta_{n,1}(\U)=0$ and $\delta_{n,i}(\U)=\text{degree of }\Delta_i(t)$.

As already mentioned, Alexander invariants of the infinite cyclic over are well-studied, so we focus on hypersurfaces with codimension 1 singularities.

\section{Stratification of the Hypersurface and Local Link Pairs}

Consider a Whitney stratification $\mathcal{S}$ of $V$, which turns $V$ into a stratified space, more precisely, a pseudomanifold. Recall that there is such a stratification where strata are pure dimensional locally closed algebraic subsets with a finite number of irreducible nonsingular components. Choose a generic hyperplane $H$ so that it is transveral to all of the strata of $V$. Then we get an induced stratification for $V\cup H$, with the strata of the form:
$$\{ X-X\cap H, X\cap H, H-V\cap H \ | \ X\in \mathcal{S}\}$$

Let $X_i^k$ be the $i$-th $k$ dimensional stratum of $V$ and let $X^k=\cup X_i^k$. For instance, if $V$ has only isolated singularities, then $X^m$ is the smooth part of $V$ and $X_i^0$ is the $i$-th singular point, while $X^k = X^0$ for $k < m$.

There is a link pair associated to each stratum $X_i^k$. Precisely, at each point of $X_i^k$, consider a small sphere $S^{2m-2k+1}$ which is the boundary of a small disk $D^{m-k+1}$ transversal to $X_i^k$ with codimension $k$ in $\C^{m+1}$. Denote $K^{2m-2k-1}: = S^{2m-2k+1} \cap V$. It is known that the link pair $(S^{2m-2k+1}, K^{2m-2k-1})$ associated to two different points on the same stratum are homeomorphic. Furthermore, these link pairs have fibered complements, by the Milnor fibration theorem.

\subsection{Local topological structure of the hypersurface complement}

In order to relate the higher-order degree $\delta_{n,i}(\U)$ and the local singularities, we look into the local topological structure of the hypersurface complement. We first reduce the problem to the study of the boundary, $X$, of a regular neighborhood of $V_a$ in $\C ^{m+1}$.
Let $N(V)$ be a small tubular neighborhood of $V$ in $\CP$ and note that, due to the transversality assumption, the complement $N(V)\setminus (V\cup H)$ can be identified with $N(V_a)\setminus V_a$, where $N(V_a)$ is a regular neighborhood of $V_a$ in $\C ^{m+1}$. Moreover $N(V)\setminus (V\cup H)$ retracts by deformation on \linebreak
$X=\partial N(V)$.

There exists a generic smooth hypersurface $L$ in $N(V)$ of dimension $m$ and of degree $d$ such that $L$ is transversal to $V\cup H$ (as in the proof of \cite{HG} theorem 4.3). By Lefschetz hyperplane section theorem (in \cite{DB}, \cite{RF}), it follows that the composition map below is an isomorphism for $i< m$ and an epimorphism for $i=m$.

$$\xymatrix{
   \pi_i(L\setminus L\cap ( V \cup H)) \ar[r] \ar[d]   &   \pi _i(\CP\setminus V \cup H) =\pi_i(\U)    \\
   \pi _i(N(V)\setminus V \cup H) \ar[ru] &   \\  
                     }$$
(the argument used here is similar to the one used in the proof of theorem 3.2 of \cite{L2})
                    
We get $\pi_i(X)\xrightarrow{\cong} \pi_i(\U)$ for $i<m$ and 
$\pi_m(X)\twoheadrightarrow \pi_m(\U)$. 
Hence $\pi_i(\U,X)=0$ for $i\leq m$. So $\U$ has the homotopy type of a complex obtained from $X$ by adding cells of dimension $\geq m+1$. The same is true for any cover, in particular for the $\Gamma_n$-covers. Since $\pi_1(X)\cong \pi_1(\U)$, all the $\Gamma_n$-covers of $X$ and $\U$ are compatible. Hence the $\Z\Gamma_n$- homomorphism $H_i(X_{\Gamma_n})\rightarrow H_i(\U_{\Gamma_n})$ is an isomorphism for $i<m$ and an epimorphism for $i=m$. 

Choose a splitting $\xi:\Z\rightarrow \Gamma_n$ to identify $R_n$ with $\Kt$ and since $\Kt$ is flat over $\Z\Gamma_n$, get
$H_i(X;\mathbb{K}_n[t^{\pm 1}])\xrightarrow{\cong} H_i(\U;\mathbb{K}_n[t^{\pm 1}])$ for $i<m$ and $H_m(X;\mathbb{K}_n[t^{\pm 1}])\twoheadrightarrow H_m(\U;\mathbb{K}_n[t^{\pm 1}])$ as $\Kt$-modules. Therefore, we have:
$$\delta_{n,i}(\U)=\text{rk}_{\mathbb{K}_n}H_i(\U;\mathbb{K}_n[t^{\pm 1}])\leq \text{rk}_{\mathbb{K}_n}H_i(X;\mathbb{K}_n[t^{\pm 1}])= \delta_{n,i}(X)$$ for all $i\leq m$. 
Hence, it is sufficient to bound above $\delta_{n,i}(X)$.

The stratification of $V_a$ with strata $X_j^k$ yields the partition of $X$ into a union of subsets $Y_j^k$. Each subset $Y_j^k$ fibers over the corresponding stratum $X_j^k$ of $V_a$, with fiber homotopy equivalent to the link complement of the stratum $X_j^k$. (The argument here is similar to the one used in \cite{EM}.)
\begin{align*}
S^{2m-2k+1}\setminus K^{2m-2k-1} \hookrightarrow   & Y_j^k \\
&\downarrow \\
 & X_j^k 
\end{align*}

\begin{rmk}\rm

1) $Y_{j_1}^k\cap Y_{j_2}^k $ is empty if $j_1\neq j_2$.

2) For $k_1< k_2$, $Y_{j_1}^{k_1}\cap Y_{j_2}^{k_2}  \neq \varnothing$ if and only if $\overline{X_{j_1}^{k_1}}\subseteq \overline{X_{j_2}^{k_2}}$.

If $k_1< k_2$, we also have a restricted fibration:
\begin{align*}
S^{2m-2{k_2}+1}\setminus K^{2m-2{k_2}-1} \hookrightarrow     Y_{j_1}^{k_1}& \cap Y_{j_2}^{k_2} \\
&\downarrow\\
&Z_{j_2}^{k_2}
\end{align*}

where $Z_{j_2}^{k_2}$ is an open subset of $X_{j_2}^{k_2}$.
\end{rmk}

\subsection{Definition of local homologies with $R_n$ coefficients and local higher-order Alexander modules}

For any subspace $M$ of $X$, we have an inclusion map $$i:M\hookrightarrow X.$$ It induces $$\pi_1(M)\xrightarrow{i_{\sharp}}  \pi_1(X),$$ and the higher-order cover $M_{\Gamma_n}$ is the cover corresponding to the kernel of $$\pi_1(M)\xrightarrow{i_{\sharp}}  \pi_1(X)\xrightarrow{\phi}\Gamma_n.$$ Define local homologies with $R_n$ coefficient as follows:

 $$H^*(M;R_n)=H_*(\text{Hom}_{\Z\Gamma_n}(C_*(M_{\Gamma_n};\Z),R_n))$$ as a left $R_n$-module, and 
 
 $$H_*(M;R_n)=H_*(C_*(M_{\Gamma_n};\Z)\otimes_{ \Z\Gamma_n}R_n)$$ as a right $R_n$-module.

In the mean while, $M$ has its own higher-order Alexander modules. Let $$\alpha_n :=\pi_1(M)/{\pi_1(M)}_r^{(n+1)},$$ we have an $\alpha_n$-cover $M_{\alpha_n}$ corresponding to $\pi_1(M) \rightarrow \alpha_n$. The map $\pi_1(M)\xrightarrow{i_{\sharp}}  \pi_1(X)$ induce a map $\alpha_n \xrightarrow{g} \Gamma_n$. Next, the linking number homomorphism guarantees that the composition $\alpha_n \xrightarrow{g} \Gamma_n \xrightarrow{\psi} \Z$ is still surjective. Let $\bar{\alpha}_n:=\text{ker}(\psi \circ g)$ and $r_n:=\Z\alpha_n(\Z\bar{\alpha}_n\setminus \{0\})^{-1}$. We can identify $r_n$ with $k_n[t^{\pm 1}]$ after choosing a splitting $\xi:\Z \rightarrow \alpha_n$. Furthermore, this splitting can be chosen to be compatible with the global one ($\xi:\Z \rightarrow \alpha_n \rightarrow \Gamma_n$).
Define $$\mathcal{A}_{n,s}(M)=H_s(M;r_n)=H_*(C_*(M_{\alpha_n};\Z)\otimes_{ \Z\alpha_n}r_n)$$ as a right $r_n$-module.

In fact, there is a relation between the covers $M_{\Gamma_n}$ and $M_{\alpha_n}$: the homomorphisms $$\pi_1(M) \twoheadrightarrow \alpha_n \xrightarrow{g} \Gamma_n \rightarrow \Z$$ yield the induced maps of covers $\sqcup M_{\alpha_n}\rightarrow M_{\Gamma_n} \rightarrow \widetilde{M} \rightarrow M$, where the number of disjoint copies of $M_{\alpha_n}$ is $[\Gamma_n : \text{im}(g)]$. Note that $\pi_1(M) \xrightarrow{\phi\circ i_{\sharp}} \Gamma_n $ might not be surjective, so \linebreak
$M_{\Gamma_n}=\sqcup M_{\text{im}(\phi\circ i_{\sharp})} $.

Here we introduce some lemmas which will help us to bound $\delta_{n,i}(X)$ above.

\begin{lem}(\cite{Di2})
Since $X=\bigcup\limits_{j,k}Y_j^k$, there is a Mayer-Vietoris spectral sequence with $R_n$ coefficients, 
$$E_1^{p,q}:=\bigoplus\limits_{k_0<...<k_p} H^q(Y_{j_0}^{k_0}\cap ...\cap Y_{j_p}^{k_p};R_n)\Longrightarrow H^{p+q}(X;R_n).$$

\end{lem}

\begin{lem}(\cite{Di2})
The Leray spectral spectral sequence for the previous fibration with  $R_n$ coefficients yields: 
$$E_2^{a,b}:=H^a(Z_{j_p}^{k_p};\mathcal{H}^b(S^{2m-2{k_p}+1}\setminus K^{2m-2{k_p}-1};R_n))\Longrightarrow H^{a+b}(Y_{j_0}^{k_0}\cap ...\cap  Y_{j_p}^{k_p};R_n).$$
for $0\leq k_0<...<k_p \leq m$, where in the $E_2$-term, $\mathcal{H}^b(-;R_n)$ is a local system of $R_n$-modules, with stalk $H^b(-;R_n)$.

\end{lem}

\begin{lem} 

$H_b(S^{2m-2{k}+1}\setminus K^{2m-2{k}-1};R_n)$ is a finite dimensional $\K$-vector space for \linebreak
$b\leq m-k$ and is trivial for $b\geq m-k+1$.
\end{lem}
\begin{proof}

Notice that for each link pair, there is a local Milnor fibration (See \cite{SP}) and its pullback under the universal covering map $\R \rightarrow S^1$.

$$\xymatrix{
 F \ar[rd]   &   & F \ar[rd] &   \\
  & S^{2m-2{k}+1}\setminus K^{2m-2{k}-1}\ar[d]&& \widetilde{ S^{2m-2{k}+1}\setminus K^{2m-2{k}-1}}\ar[d] \\  
 & S^1&&\widetilde{S^1}\simeq \mathbb{R}\\ 
                     }$$
                 
where $F$ is the Milnor fiber. 
Since $\mathbb{R}$ is contractible, the second fibration is trivial. i.e. $\widetilde{ S^{2m-2{k}+1}\setminus K^{2m-2{k}-1}} \simeq F\times \mathbb{R}\simeq F$.
By \cite{SP}, the Milnor fiber $F$ has the homotopy type of a CW complex of dimension $m-k$, so $H_b(F;-)=0$ for all $b\geq m-k+1$. Also since the $\Gamma_n$-cover of $S^{2m-2{k}+1}\setminus K^{2m-2{k}-1}$ factors through the infinite cyclic cover, $$H_b(S^{2m-2{k}+1}\setminus K^{2m-2{k}-1};R_n)\cong H_b(\widetilde{ S^{2m-2{k}+1}\setminus K^{2m-2{k}-1}};\K)\cong H_b(F;\K),$$ and the latter is a finite dimensional $\K$-vector space for $b\leq m-k$ and it is $0$ for $b\geq m-k+1$.

\end{proof}

\begin{lem} (\cite{Ex}, \cite{Mu}) There exists a spectral sequence:
$$E_2^{s,l}:=\text{Ext}_{r_n}^l (\mathcal{A}_{n,s}(S^{2m-2{k_p}+1}\setminus K^{2m-2{k_p}-1}),R_n)\Longrightarrow H^{s+l}(S^{2m-2{k_p}+1}\setminus K^{2m-2{k_p}-1};R_n).$$

\end{lem}

\section{Computing the EXT Terms}

Recall that a pair of compatible splittings $\Z\rightarrow \Gamma_n$ and $\Z\rightarrow \alpha_n$ can be chosen such that they match the map $\alpha_n\xrightarrow{g} \Gamma_n \rightarrow \Z$. So $g$ restricts to a map $\bar{\alpha}_n \xrightarrow{g}\bar{\Gamma}_n$. It induces a map between the coefficients $$k_n=\Z\bar{\alpha}_n(\Z\bar{\alpha}_n\setminus \{0\})^{-1} \xrightarrow{g} \mathbb{K}_n=\Z\bar{\Gamma}_n(\Z\bar{\Gamma}_n\setminus \{0\})^{-1}.$$
To understand the Ext terms in Lemma 6.5, our goal is to compute Ext$_{r_n}^l(\mathcal{A}_{n,s}(S\setminus K); R_n)$. Since $\mathcal{A}_{n,s}(S\setminus K)$ is a torsion $k_n[t^{\pm 1}]$-module, $$\mathcal{A}_{n,s}(S\setminus K)=\oplus_{i=1}^e \frac{k_n[t^{\pm 1}]}{(p_i(t))}$$ for some polynomials $p_i(t)$. So we need to compute Ext$_{k_n[t^{\pm 1}]}^l(k_n[t^{\pm 1}]/{(p(t))};\Kt)$ for some polynomial $p(t)$.

\begin{fact}
Since $k_n[t^{\pm 1}]$ is a free $k_n[t^{\pm 1}]$ module, it is projective.

\end{fact}

\begin{fact}

For a projective module $R$, Ext$_R^l(R, B)=0$ for all $B$ and for $l\geq 1$.

\end{fact}

For simplicity, denote $A=k_n[t^{\pm 1}]/{(p(t))}$, $B=\Kt$ and $R=k_n[t^{\pm 1}]$. 
Consider the short exact sequence $0 \rightarrow \kt \xrightarrow{p(t)} \kt \rightarrow \ktp \rightarrow 0$, which is a projective resolution of $A$. The homology long exact sequence gives 
\begin{align*}
\text{Hom}_R(\ktp, B) &\rightarrow \text{Hom}_R(\kt, B) \xrightarrow{\tilde{p}(t)} \text{Hom}_R(\kt, B)\\
\rightarrow \text{Ext}_R^1(\ktp, B) &\rightarrow \text{Ext}_R^1(\kt, B) \rightarrow\text{Ext}_R^1(\kt, B)\\
\rightarrow \text{Ext}_R^2(\ktp, B) &\rightarrow \text{Ext}_R^2(\kt, B) \rightarrow \text{Ext}_R^2(\kt, B)\rightarrow ...
\end{align*} 
where $\tilde{p}(t):=g(p(t))$. Since $\ktp$ is a torsion $\kt$-module and $B=\Kt$ is a free $\Kt$-module, we have that $\text{Hom}_R(\ktp, B)=0$. 

\begin{lem}
$\text{Hom}_{\kt}(\kt,\Kt)\cong \Kt$ as $\Kt$-modules.
\end{lem}

\begin{proof}
Recall that $\kt$ and $\Kt$ are Laurent polynomial rings over skew fields and $g$ can be extended to a ring map $g:\kt \rightarrow \Kt$ which sends $1$ to $1$ and $t$ to $t$.

Define $\Phi :\Kt \rightarrow \text{Hom}_{\kt}(\kt, \Kt)$ by sending $b(t)$ to $$\alpha :=\Phi(b(t)) :\kt \rightarrow \Kt$$ where $\alpha (a(t))=b(t)\cdot g(a(t))$. 

Now we show $\alpha $ is an element of $\text{Hom}_{\kt}(\kt, \Kt)$:
\begin{align*}
\alpha (a(t)+a'(t))& =b(t)(g(a(t))+g(a'(t))) \\
                & =b(t)g(a(t))+b(t)g(a'(t))\\
                & = \alpha(a(t))+\alpha (a'(t)).
\end{align*}

For any $r(t) \in \kt$, 
\begin{align*}
\alpha(a(t)r(t))& =b(t)g(a(t)r(t)) \\
                & =b(t)g(a(t))g(r(t))\\
                & =\alpha(a(t))g(r(t)). 
\end{align*}

Then define $\Psi : \text{Hom}_{\kt}(\kt, \Kt) \rightarrow \Kt$ by sending $\alpha$ to $\Psi (\alpha)=\alpha (1)$.
So \begin{align*}
\Psi \circ \Phi (b(t)) &=\Phi (b(t))(1) \\
  & =b(t)g(1)\\
  &=b(t) .           
\end{align*}
and $\Phi \circ \Psi(\alpha) =\Phi (\alpha (1))$, 
where  \begin{align*}
\Phi (\alpha (1))(a(t))&=\alpha(1)g(a(t))\\
& =\alpha (1  a(t))\\
&=\alpha(a(t)).           
\end{align*}
Since $\alpha$ is a $\kt$-module homomorphism. Thus $\Phi \circ \Psi(\alpha) =\alpha$ and we conclude that
$\text{Hom}_{\kt}(\kt,\Kt )\cong \Kt$.

\end{proof}

Thus the above long exact sequence becomes
$$0 \rightarrow 0 \rightarrow \Kt \xrightarrow{\tilde{p}(t)} \Kt \rightarrow \text{Ext}_R^1(A, B) \rightarrow 0 \rightarrow 0 \rightarrow \text{Ext}_R^2(A, B)\rightarrow 0\rightarrow ...,$$

Since $\tilde{p}(t)=g(p(t))$, it is clear that the degree of $\tilde{p}(t)$ is less or equal to the degree of $p(t)$.

Therefore, \begin{align*}
&\text{Ext}_R^0(A, B)=\text{Hom}_R(A, B)=0,\\
&\text{Ext}_R^1(A, B)\cong \Kt/{(\tilde{p}(t))},\\
&\text{Ext}_R^{\geq 2}(A, B)=0.           
\end{align*}

Recall in lemma 6.5, $$E_2^{s,l}:=\text{Ext}_{r_n}^l (\mathcal{A}_{n,s}(S\setminus K),R_n)\Longrightarrow H^{s+l}(S\setminus K;R_n).$$ Since $$\text{Ext}_R^l(\oplus_{i=1}^e \frac{k_n[t^{\pm 1}]}{(p_i(t))}, \Kt)=\prod_{i=1}^e\text{Ext}_R^l( \frac{k_n[t^{\pm 1}]}{(p_i(t))}, \Kt),$$ and Ext$_R^l(\ktp, \Kt)=0$  for $l\neq 1$, Ext$_R^l(\mathcal{A}_{n,s}(S\setminus K),R_n)=0$ for $l\neq 1$. As a result, $$H^{s+1}(S\setminus K;R_n)\cong \text{Ext}_{r_n}^1 (\mathcal{A}_{n,s}(S\setminus K),R_n).$$

Therefore, 
\begin{align*}
\text{rk}_{\mathbb{K}_n}H^{s+1}(S\setminus K;R_n)&= \text{rk}_{\mathbb{K}_n}\text{Ext}_{r_n}^1 (\mathcal{A}_{n,s}(S\setminus K),R_n)\\
&=\text{rk}_{\mathbb{K}_n}\text{Ext}_{\kt}^1 (\oplus_{i=1}^e \frac{k_n[t^{\pm 1}]}{(p_i(t))},\Kt)\\
&=\sum\limits_{i=1}^e \text{rk}_{\mathbb{K}_n}\frac{\Kt}{(\tilde{p_i}(t))} \\
&\leq \sum\limits_{i=1}^e  \text{rk}_{k_n}\frac{\kt}{(p_i(t))} \\
&=\delta _{n,s}(S\setminus K) .\\
\end{align*}
We summarize our calculations as a lemma.

\begin{lem}
With above notations, we have $\text{rk}_{\mathbb{K}_n}H^{s+1}(S\setminus K;R_n)\leq \delta _{n,s}(S\setminus K) $.

\end{lem}

\section{Upper Bounds for the Higher-Order Degrees by the Local Invariants}

With the above calculation, we are ready to prove the following result:

\begin{theorem}\label{local-bound}

Let $V$ be a reduced hypersurface in $\CP$ with generic hyperplane $H$ at infinity and let $\U=\CP\setminus V \cup H$. Then for $i\leq m$:

$$\delta_n^i(U) \leq
\sum\limits_{p+q=i+1,} \sum\limits_{0\leq k_0< k_1,...< k_p\leq m,}  \sum\limits_{a+b=q} \theta(k_p, a)\cdot \delta _{n,{b-1}}(S^{2m-2{k_p}+1}\setminus K^{2m-2{k_p}-1})$$ for $b \geq 1$, 
where $\theta(k_p, a)$ is the number of $a$-dimensional cells of $W^{k_p} $, a finite CW complex homotopy equivalent to $Z_{j_p}^{k_p}$ (which is described in section 6).

\end{theorem}

\begin{proof}
In section 6, we have shown that 
$$\delta_{n,i}(\U)=\text{rk}_{\mathbb{K}_n}H_i(\U;\mathbb{K}_n[t^{\pm 1}])\leq \text{rk}_{\mathbb{K}_n}H_i(X;\mathbb{K}_n[t^{\pm 1}]).$$ 
By Lemma 6.1, 
\begin{align*}
\text{rk}_{\mathbb{K}_n}H^i(X;\mathbb{K}_n[t^{\pm 1}])&\leq \sum\limits_{p+q=i,} \sum\limits_{0\leq k_0<...<k_p\leq m} \text{rk}_{\K} H^q(Y_{j_0}^{k_0}\cap ...\cap Y_{j_p}^{k_p};R_n)\\
\text{(by Lemma 6.2) }                &\leq 
\sum\limits_{p+q=i,} \sum\limits_{0\leq k_0< ...< k_p\leq m,}  \sum\limits_{a+b=q} \text{rk}_{\K} H^a( Z_{j_p}^{k_p};\mathcal{H}^b(S^{2m-2{k_p}+1}\setminus K^{2m-2{k_p}-1};R_n)).
\end{align*}

Note that $Z_{j_p}^{k_p}$ is homotopy equivalent to a finite $k_p$-dimensional CW complex $W^{k_p}$. Let $\theta (k_p, a)$ be the number of $a$-dimensional cells of $W^{k_p} $,
then the inequality becomes
$$\text{rk}_{\mathbb{K}_n}H^i(X;\mathbb{K}_n[t^{\pm 1}])\leq
\sum\limits_{p+q=i,} \sum\limits_{0\leq k_0< k_1,...< k_p\leq m,}  \sum\limits_{a+b=q} \theta(k_p, a)\cdot \text{rk}_{\K}H^b(S^{2m-2{k_p}+1}\setminus K^{2m-2{k_p}-1};R_n).$$ 

By lemma 6.3, we know that $H_b(S^{2m-2{k_p}+1}\setminus K^{2m-2{k_p}-1};R_n)$ is a finite dimensional $\K$-vector space, then by UCT, so is $H^b(S^{2m-2{k_p}+1}\setminus K^{2m-2{k_p}-1};R_n)$. Hence $\text{rk}_{\mathbb{K}_n}H^i(X;\mathbb{K}_n[t^{\pm 1}])$ is finite.

Apply UCT again, 
$H^i(X;\mathbb{K}_n[t^{\pm 1}])=\text{free}H_i(X;\mathbb{K}_n[t^{\pm 1}])\oplus\text{torsion}H_{i-1}(X;\mathbb{K}_n[t^{\pm 1}])$, and since we have shown that $\text{rk}_{\mathbb{K}_n}H^i(X;\mathbb{K}_n[t^{\pm 1}]) $ is finite, so is $\text{rk}_{\mathbb{K}_n}H_i(X;\mathbb{K}_n[t^{\pm 1}]) $.

 Therefore,
\begin{align*}
\delta_{n,i}(\U)&\leq rk_{\mathbb{K}_n}H_i(X;\mathbb{K}_n[t^{\pm 1}])=\text{rk}_{\mathbb{K}_n}H^{i+1}(X;\mathbb{K}_n[t^{\pm 1}])\\
&\leq
\sum\limits_{p+q=i+1,} \sum\limits_{0\leq k_0< k_1,...< k_p\leq m,}  \sum\limits_{a+b=q} \theta(k_p, a)\cdot \text{rk}_{\mathbb{K}_n}H^{b}(S^{2m-2{k_p}+1}\setminus K^{2m-2{k_p}-1};R_n)\\
&\leq
\sum\limits_{p+q=i+1,} \sum\limits_{0\leq k_0< k_1...< k_p\leq m,}  \sum\limits_{a+b=q} \theta(k_p, a)\cdot\delta _{n,{b-1}}(S^{2m-2{k_p}+1}\setminus K^{2m-2{k_p}-1})
\end{align*}
 for $b \geq 1$, where the last inequality follows by lemma 6.5 and lemma 7.4.

\end{proof}

According to the theorem, we see that $\delta_{n,0}(\U)$ depends on $\delta_{n,0}(S\setminus K)$ where $S\setminus K$ denotes the link complements of the various strata of $V\setminus H$, while $\delta_{n,1}(\U)$ depends on $\delta_{n,0}(S\setminus K)$ and $\delta_{n,1}(S\setminus K)$. In fact, the global degree depends on a smaller range of the degrees of local links.

\begin{cor}
For $0\leq i\leq m$, $\delta_{n,i}(\U)$ depends on $  \delta _{n,c}(S^{2m-2{k}+1}\setminus K^{2m-2k-1})$ with the ranges

$m-i\leq k\leq m$ and

$3m-3k-2i\leq c\leq m-k$.

\end{cor}

\begin{proof}

\textbf{Step 1}: We need a lemma

\begin{lem}
For $0\leq i\leq m$, $\delta_{n,i}(\U)$ depends on $  \delta _{n,{b-1}}(S^{2m-2{k}+1}\setminus K^{2m-2k-1})$ for the range

$0\leq k\leq m$ and

$i-3k \leq b-1\leq m-k$.

\end{lem}

According to the proof of the theorem,
$$\delta_{n,i}(\U) \leq
\sum\limits_{p+q=i+1,} \sum\limits_{0\leq k_0< k_1,...< k_p\leq m,}  \sum\limits_{a+b=q}H^a( Z_{j_p}^{k_p};H^b(S^{2m-2{k_p}+1}\setminus K^{2m-2{k_p}-1};R_n))$$ for $b \geq 1$,
,
 here $p+1$ is the number of intersection of $Y$'s with each other, $k_p\geq p$, $p+q=i+1$ and $a+b=q$.

By lemma 6.3 and UCT, $H^b(S^{2m-2{k}+1}\setminus K^{2m-2{k}-1};R_n)=H_{b-1}(S^{2m-2{k}+1}\setminus K^{2m-2{k}-1};R_n)$ is a finite dimensional $\K$-vector space for $b-1\leq m-k$ and is $0$ for $b-1\geq m-k+1$.
So we have $0\leq b-1\leq m-k$.

Also $H^a( Z_{j_p}^{k_p};-)$ is nontrivial only when $0\leq a\leq 2 k_p$ since $Z_{j_p}^{k_p}$ is a complex $ k_p$-dimensional manifold.  Thus we have 
$p+a+b=p+q=i+1$ and so $0\leq a=i+1-p-b \leq 2 k_p$. 

Since $p\leq k_p$, we have $i-3k_p\leq b-1\leq m-k_p$. Forget the index of intersections and write $k$ for $k_p$. Then, $0\leq k\leq m$ and
$i-3k \leq b-1\leq m-k$.

\bigskip

\textbf{Step 2}: Apply Lefschetz hyperplane section theorem.

Let $1\leq i=:m-j\leq m$ be fixed. We are looking for the range of $k$ and $b-1$. Let $L\cong \mathbb{CP}^{m-j+1}$ be a generic
codimension j linear subspace of $\CP$ transversal to all strata of $V\cup H$. By transversality, $V\cap L$ is an $i$ dimensional, degree $d$, reduced hypersurface in $L$. And 
it is transversal to the hyperplane at infinity $H \cap L$. We consider the Whitney stratification of the pair $(L, V\cap L)$ which is induced from the pair
$(\CP, V)$. The strata in $V\cap L$ are the strata in $V$ intersecting $L$.

By Lefschetz hyperplane section theorem, $\U\cap L \rightarrow \U$ is an $m-j+1$ equivalence, i.e., the homotopy type of $\U$ is obtained from $\U\cap L$ by adding cells of dimension greater than $m-j+1$. So we have isomorphisms
 $H_l(\U\cap L;\Kt)\rightarrow H_l(\U;\Kt)$ for $l\leq m-j=i$. Note that 
 $H_i(\U\cap L;\Kt)$ is the top torsion higher-order Alexander module of $\U\cap L$. So we can apply the Lemma 8.3 in step 1 to it.
In this case, $\U\cap L$ has dimension $m-j$ and the link pair of stratum $S\cap L$ in $V\cap L$ is the same as the link pair of the stratum $S$ in $V$, while dim$S\cap V=\text{dim}S-j$. So if dim$S=k$, then dim$S\cap L=k-j$. 

Apply Lemma 8.3 on $\U\cap L$, get $0\leq k-j\leq m-j$ and $m-j-3(k-j) \leq b-1\leq m-j-(k-j)$.

Therefore, we get 
$m-i\leq k\leq m$
and $m+2j-3k=3m-3k-2i\leq c\leq m-k$ where $c=b-1$.

\end{proof}

Recall that if we choose a splitting $\xi$ to identify $R_n$ with $\mathbb{K}_n[t^{\pm 1}]$, we can define \newline
$\mathcal{A}_{n,i}^{\xi}(\U) = H_i(\U;\mathbb{K}_n[t^{\pm 1}])$. 

\begin{rmk}\rm
If we choose a splitting $\xi$ to identify $R_n$ with $\mathbb{K}_n[t^{\pm 1}]$, the higher-order Alexander polynomial $p_{n,i}^{\xi}(t)$ is the order of $H_i(\U;\mathbb{K}_n[t^{\pm 1}])$.

 Then the $n$-th order degree $\delta_{n,i}(\U)$ is the degree of $p_{n,i}^{\xi}(t)$.
\end{rmk}

 Recall that the higher-order Alexander polynomial depends on the choice of splitting. However, the degree of a higher-order Alexander polynomial is the same regardless of the choice of splitting, so the higher-order degree is well-defined in general.

\begin{cor} Fix a splitting $\xi$ to identify $R_n$ with $\Kt$, which matches all the local coefficients $\kt$, that is, there is a map $\K \xrightarrow{g} k_n$ as explained in section 7.

Then the prime factors of $H_i(\U,\Kt)$ (i.e. the roots of the higher-order Alexander polynomial $p_{n,i}^{\xi}(t)$ of  \ $\U$) are the roots of $g(q_{n,c}^{\xi}(t))$ where $q_{n,c}^{\xi}(t)$ are the higher-order Alexander polynomials of 
the local links $S^{2m-2{k}+1}\setminus K^{2m-2{k}-1}$ and $g:\kt \rightarrow \Kt$ maps the coefficients with t unchanged. The ranges for $k$ and $c=b-1$ are given $m-i\leq k\leq m$ and $3m-3k-2i\leq c\leq m-k$.

\end{cor}

\begin{proof}
From the discussion in section 6, the prime factors of $H_i(\U;\mathbb{K}_n[t^{\pm 1}])$ are among the prime factors of $H_i(X;\mathbb{K}_n[t^{\pm 1}])$.
By Lemma 6.1, the prime factors of $H_i(X;\mathbb{K}_n[t^{\pm 1}])$ are among the prime factors of 
$ H^q(Y_{j_0}^{k_0}\cap ...\cap Y_{j_p}^{k_p};R_n)$ where $0\leq k_0 < ...< k_p \leq m$ and $p+q=i+1$.

By Lemma 6.2, the prime factors of $ H^q(Y_{j_0}^{k_0}\cap ...\cap Y_{j_p}^{k_p};R_n)$ are among the prime factors of $H^b(S^{2m-2{k_p}+1}\setminus K^{2m-2{k_p}-1};R_n)$ for $b\leq q$. 

By Lemma 6.3 and the calculations in section 7, the prime factors of $H^b(S^{2m-2{k_p}+1}\setminus K^{2m-2{k_p}-1};R_n)$ are among the prime factors of $\text{Ext}_{r_n}^1 (\mathcal{A}_{n,b-1}(S\setminus K),R_n)$, which corresponds to the roots of $g(q_{n,b-1}^{\xi}(t))$ where $q_{n,b-1}^{\xi}(t)$ are the higher-order Alexander polynomials of 
the local links $S^{2m-2{k}+1}\setminus K^{2m-2{k}-1}$, and $g:\kt \rightarrow \Kt$ maps the coefficients with t unchanged. 
At last, by corollary 8.3, the ranges for $k$ and $c$ are $m-i\leq k\leq m$ and $3m-3k-2i\leq c\leq m-k$.

\end{proof}

\section{Compute Higher-Order Degrees $\delta_n(U)$ Using Fox's Free Calculus}

Since $\mathcal{A}_{n,1}^{\Z}(\U)$ depends on the fundamental group $G$ only, whereas the higher-dimensional $\mathcal{A}_{n,i}^{\Z}(\U)$ also depends on the cell-structure of the cover $\U_{\Gamma_n}$, $\delta_{n,1}(\U)$ (denoted as $\delta_n(\U)$) can be calculated based on the fundamental group, using the method of Fox calculus (See \cite{FC}).

Let $\U$ be a hypersurface complement and suppose that $G=\pi_1(\U)$ is a finitely presented group with presentation 
$$G=\langle x_1,x_2,...,x_k \ | \ r_1,...r_l\rangle,$$  Let
$F=\langle x_1,x_2,...,x_k \rangle$ be the free group on $k$ generators and $\chi :F \twoheadrightarrow G$. For each $x_i$ we have a mapping called the $i$-th free derivative $\frac{\partial}{\partial x_i}:F \rightarrow \Z F$, and it is determined by 

$\frac{\partial x_i}{\partial x_j}=\delta_{i,j}$   and $\frac{\partial (uv)}{\partial x_i}=\frac{\partial u}{\partial x_i}+u\frac{\partial v}{\partial x_i}$. It is then follows that $\frac{\partial u^{-1}}{\partial x_i}=-u^{-1}\frac{\partial u}{\partial x_i}$.

The map $\chi: F\rightarrow G$ extends by linearity to a map $\chi: \Z F\rightarrow \Z G$. The matrix
\begin{scriptsize}
\[ H_1(\U,u_0;\Z G)=
\left(\frac{\partial r_i}{\partial x_j}\right)_{i,j}^{\chi} 
= \left( \begin{array}{cccc}
\chi \left(\frac{\partial r_1}{\partial x_1}\right) & \chi \left(\frac{\partial r_1}{\partial x_2}\right) & ... & \chi \left(\frac{\partial r_1}{\partial x_k}\right) \\
\chi \left(\frac{\partial r_2}{\partial x_1}\right) & \chi \left(\frac{\partial r_2}{\partial x_2}\right)  & ... & \chi \left(\frac{\partial r_2}{\partial x_k}\right) \\
\vdots & \vdots & \ddots & \vdots\\
\chi \left(\frac{\partial r_l}{\partial x_1}\right) & \chi \left(\frac{\partial r_l}{\partial x_2}\right)  & ... & \chi \left(\frac{\partial r_l}{\partial x_k}\right) 
 \end{array} \right)
\]
\end{scriptsize}
is called the Jacobian of the presentation $P$. This matrix is dependent on the presentation.
Here the columns correspond to the generators $x_j$, and the rows correspond to relations $r_i$.

If we allow elements of $G_r^{(n+1)}$ to be set equal to 1 in $\Z G$, we can also consider the above as a presentation matrix for $ H_1(\U,u_0;\Z \Gamma_n)$. Furthermore, since $R_n$ is a flat $\Z \Gamma_n$-module, we can also consider it to be a presentation matrix for $ H_1(\U,u_0;R_n)$. If we think of the matrix in this way, every non-zero element in $\Z \bar{\Gamma}_n$ has an inverse.

If we choose a splitting $\xi :\Z \rightarrow  \Gamma_n$, we can identify $R_n$ with $\Kt$. To obtain a presentation for $H_1(\U,u_0;\Kt)$ we must replace each entry in the above matrix with its image under the isomorphism $R_n\rightarrow \Kt$. Since $\Kt$ is a non-commutative ring, we must be careful when writing elements where $t$ is not originally on the right. The splitting $t\rightarrow x_1$ induces an automorphism of $\bar{\Gamma}_n$ by $g \rightarrow t^{-1}gt=g^t \in \bar{\Gamma}_n$. For example, $x_1 x_2$ will become $tx_2=x_1x_2x_1^{-1}t$. 
The next step to find $\delta_n(\U)$ is diagonalizing the matrix, which is possible since $\Kt$ is a PID.

\begin{lem} 

Here are some valid operations:(\cite{Ha}, Lemma 9.2)\rm 

1) Multiply columns on the right. Note that the matrix is a presentation of a left module and columns correspond to generators.

2) Interchange two columns.

3) Add a constant multiple of one column (multiply on the right) to another column.

4) Multiply rows on the left. Note that rows correspond to relations.

5) Interchange two rows.

6) Add a constant multiple of one row (multiply on the left) to another row.

7)\begin{tiny}
$ P \Leftrightarrow 
\left( \begin{array}{cc}
P \\
0
\end{array} \right)
$
\end{tiny} that is if we have one row only consists of 0's, we can eliminate that row.

8)\begin{tiny}
$ P \Leftrightarrow 
\left( \begin{array}{cc}
P & 0\\
* & 1
\end{array} \right)
$\end{tiny} this means if we have one column with only one entry is 1 but others are 0, then we can eliminate that row and column.

9) \begin{tiny}
$ P \Leftrightarrow 
\left( \begin{array}{cc}
P & *\\
0 & 1
\end{array} \right)
$\end{tiny} this means if we have one row with only one entry is 1 but others are 0, then we can eliminate that row and column.
\end{lem}

\begin{fact} $1-a$ is invertible in $\Kt$ for all $a \neq 1$ in $G/G'$. 
\end{fact}

\begin{proof}

Since $a \neq 1 \in G/G'=H_1(\U;\Z)$, $a \notin G'$. Recall that for hypersurface complements, $H_1(\U;\Z)$ has no $\Z$-torsion. Thus $G/G_r'=G/G'$ and we have
$G_r^{(n)}\subseteq G_r^{(n-1)}\subseteq ...\subseteq G_r^{(1)}=G'$. It follows that $a \notin G_r^{(n)}$ for
all $n \geq 1$. Therefore $a \neq 1$ in $\Gamma_n$ for all $n\geq 0$. Hence $1-a \neq 0$ in $\Z\bar{\Gamma}_n\subseteq \Z\Gamma_n$ and is
therefore invertible in $\mathbb{K}_n[t^{\pm 1}]$. This allows us to divide any column or row by the unit $1 -a$. 

\end{proof}

Similarly, if $a \neq 1$ in $G/G_r^l$ for some $l\leq n$, then $1-a$ is invertible in $\Kt$

\begin{theorem} (\cite{J}, Theorem 16, p.43)
The Jacobian matrix is equivalent to a diagonal presentation matrix of the form:

\begin{tiny}

\[ \left( \begin{array}{ccccc}
p_1(t) &0 & ...& 0& 0\\
0 & p_2(t)&...   & 0& 0\\
\vdots &\vdots&\ddots &\vdots &\vdots\\
0& 0 &\cdots &p_{u}(t)& 0
\end{array} \right)
\]\end{tiny}
\end{theorem} 

Notice that the last column can not be eliminated. We obtain $$ H_1(\U,u_0;\Kt)\cong \Kt \oplus \frac{\Kt}{p_1(t)}\oplus \frac{\Kt}{p_2(t)}\oplus\cdots \oplus \frac{\Kt}{p_u(t)}.$$
To find $H_1(\U; \mathbb{K}_n[t^{\pm 1}])$,
consider the long exact sequence of a pair:
$$ 0\rightarrow H_1(\U; \mathbb{K}_n[t^{\pm 1}])\rightarrow H_1(\U, u_0; \mathbb{K}_n[t^{\pm 1}])\rightarrow H_0( u_0; \mathbb{K}_n[t^{\pm 1}]).$$

Since $ H_1(\U; \mathbb{K}_n[t^{\pm 1}])$ is a torsion module and $ H_0( u_0; \mathbb{K}_n[t^{\pm 1}])$ is a free module,
we conclude that $$ H_1(\U; \mathbb{K}_n[t^{\pm 1}])\cong \frac{\Kt}{p_1(t)}\oplus \frac{\Kt}{p_2(t)}\oplus\cdots \oplus \frac{\Kt}{p_u(t)}.$$ Therefore, the higher-order degree is  $$\delta_n(\U) =\text{deg} \prod\limits_{1\leq i\leq u}p_i(t).$$

There are some other ways to calculate the higher-order degrees, let us illustrate that by showing an example.
\begin{eg}\rm
$$G=\langle a, b \ | \ aba^{-1}b^{-1}\rangle \cong \Z \times \Z$$
 is the fundamental group of an affine curve complement, where the curve is two lines intersecting transversally. $\delta_n=0$ for all $n\geq 0$.

\end{eg}
\begin{proof}

\textbf{Method 1}: (by definition of $H_1(\U;\Z \Gamma_n)$)

$G'=[G,G]=1$, and since $G$ has no torsion elements, $G_r^{(1)} =G'=1$.
Since $G_r^{(1)}\supseteq G_r^{(n+1)}$ for all $n\geq 0$, $G_r^{(n+1)}=1$ .
$$H_1(\U;\Z \Gamma_n)=G_r^{(n+1)}/[G_r^{(n+1)},G_r^{(n+1)}]=0,$$ so $H_1(\U;\Kt)=0$ and $\delta_n=0$ for all $n\geq 0$.

\bigskip
\textbf{Method 2}: (by the following proposition)

\begin{prop} (\cite{HP}, proposition 3.10)
Suppose $C$ is defined by a weighted homogeneous polynomial $f(x,y)=0$ in $\C ^2$, and assume that either $n>0$ or $\beta _1(U)>1$. Then we have:
$$\delta_n=\mu(C,0)-1,$$ where $\mu(C,0)$ is the Milnor number associated to the singularity germ at the origin. If $\beta _1(\U)>1$, then $\delta_0=\mu(C,0)$.

\end{prop}

Notice that $G$ is the fundamental group of the complement of the curve $C=\{x^2+y^2=0\}$.  $H_1(\U)\cong \Z \times \Z$, so $\beta _1(\U)>1$. By the proposition, $\delta_n=\mu(C,0)-1$. On the other hand, for a weighted homogeneous polynomial with degree $d$ and $s$ variable with weights $w_1, ..., w_s$, the Milnor number is (See \cite{DB})
$$\mu(C,0)=\frac{(d-w_1)\cdots(d-w_s)}{w_1\cdots w_s}.$$ In our case, $d=2$, $w_1=w_2=1$, so $\mu(C,0)=\frac{(2-1)(2-1)}{1\cdot 1}=1$. Thus $\delta_n=1-1=0$.

\bigskip
\textbf{Method 3}: (by Fox Calculus)

For simplicity, we find a new presentation of $G$ such that one generator has linking number one and the others are zero. Let $y=ba^{-1}$, that is $b=yx$. The relation $aba^{-1}b^{-1}=xyxx^{-1}x^{-1}y^{-1}=xyx^{-1}y^{-1}$ stays the same. So $G=\langle x, y \ | \ xyx^{-1}y^{-1}\rangle $.
Let $r_1= xyx^{-1}y^{-1}$, then $\frac{\partial r_1}{\partial x}=1-y$ and $\frac{\partial r_1}{\partial y}=x-1$. Hence, we get a presentation for $H_1(\U,u_0;\Z G)$ as a left $\Z G$-module.

$$\left( \begin{array}{cc}
1-y & x-1
\end{array} \right).$$

By the argument in method 1, $G_r^{(n+1)}=1$ for all $n\geq 0$, so $y \notin G_r^{(n+1)}$ and $y\neq 1$ in $\Gamma_n=G/G_r^{(n+1)}$. Thus $1-y\neq 0$ in $\Z \Gamma_n$ and is invertible in $\mathbb{K}_n$ for all $n\geq 0$.

To obtain a presentation for $H_1(\U,u_0;\Kt)$, we choose the splitting that maps $t$ to $x$. Then we have
$$\left( \begin{array}{cc}
1-y & t-1
\end{array} \right)$$

Since $1-y$ is invertible, it's equivalent to 
$$\left( \begin{array}{cc}
1 & t-1
\end{array} \right)
\thicksim
\left( \begin{array}{cc}
1 & 0
\end{array} \right)$$
Hence $$H_1(\U, u_0;\Kt)=\Kt\oplus \Kt/(1)=\Kt,$$ and so $H_1(\U;\Kt)=0$ and $\delta_n=0$ for $n\geq 0$.
\end{proof}

\begin{eg}\rm

The higher-order degrees of the complement of $m$ lines intersecting at a single point in $\C^2$ are $$\delta_n=m(m-2).$$

\end{eg}

Let $C$ the curve formed by $m$ lines intersecting at a single point in $\C^2$ and $\U=\C^2-C$.

\begin{claim}
$\pi_1(\U)=\mathbb{F}_{m-1}\times \Z$ where $\mathbb{F}_{m-1}$ is the free group of $m-1$ generators.
\end{claim}

\begin{proof}

Consider the following restriction of Hopf fibration $$\C^* \hookrightarrow \C^2-C \rightarrow \mathbb{P}^1-\{m\text{ points}\}.$$ 

This fibration can be embedded in a trivial fibration, hence it is also trivial.
$$\xymatrix{
             \C^* \ar[rd]                &  &        &    \C^* \ar[rd]          & \\
     &   \C^2-C \ar[d]   &    \hookrightarrow&             &  \C^2- \{\text{1 line} \ar[d]\} \\  
     & \mathbb{P}^1 -\{m\text{ points}\}   &        &             &  \mathbb{P}^1 - \{1\text{ point}\} \simeq\C\\                            
}$$

Thus $$\C^2-C\simeq\ C^*\times (\mathbb{P}^1 -\{m\text{ points}\} )\simeq S^1\times (S^1 \vee \cdots \vee S^1)$$ and $$\pi_1(\U)\cong \pi_1(S^1)\times \pi_1(S^1 \vee \cdots \vee S^1)\cong  \Z \times(\Z*\cdots *\Z).$$
\end{proof}
$$G=\pi_1(\U)=\langle a_1,..., a_{m-1},y \ | \ a_1ya_1^{-1}y^{-1},...,a_{m-1}ya_{m-1}^{-1}y^{-1}\rangle.$$
 In this presentation, $a_i$'s are the meridians around the line components of $C$, and $y$ is the meridian around the line at infinity. So $lk(a_1)=...=lk(a_{m-1})=1$ and $lk(y)=m$.

For simplicity, let $x_1=a_1$, $x_2=a_2a_1^{-1},...,x_{m-1}=a_{m-1}a_1^{-1}$, get
$$G=\langle x_1,..., x_{m-1},y \ | \ x_1yx_1^{-1}y^{-1},...,x_{m-1}yx_{m-1}^{-1}y^{-1}\rangle$$ with $lk(x_2)=...=lk(x_{m-1})=0$. An presentation matrix for $H_1(\U,u_0;\Z G)$, as a left $\Z G$-module, is a $(m-1)\times m$ matrix.
\begin{tiny}

$$\left( \begin{array}{ccccc}
1-y & 0 &\cdots &0& x_1-1\\
0 & 1-y &\cdots &0& x_2-1\\
\vdots & \vdots &\ddots&\vdots& \vdots\\
0 & 0 &\cdots &1-y& x_{m-1}-1\\
\end{array} \right)$$
\end{tiny}

Since $H_1(\U)=G/G'\cong \Z^m$ is a free abelian group with generators $x_1,...,x_{m-1},y$, in particularly $x_{m-1}\neq 1$ in $G/G'$. So by fact 9.2, $1-x_{m-1}$ is a unit in  $\mathbb{K}_n$.

Multiply the $(m-1)$-th column by $1-x_{m-1}$ on the right. Add the first column times $1-x_1$, the second column times $1-x_2$, ... , the $(m-2)$-th column times $1-x_{m-2}$, and the $m$-th column times $1-y$ to the $(m-1)$-th column, 
\begin{tiny}

$$\left( \begin{array}{ccccc}
1-y & 0 &\cdots &0& x_1-1\\
0 & 1-y &\cdots &0& x_2-1\\
\vdots & \vdots &\ddots&\vdots& \vdots\\
0 & 0 &\cdots &0& x_{m-1}-1\\
\end{array} \right)$$
\end{tiny}

Interchange the last two columns, and muliply the last row by $(x_{m-1}-1)^{-1}$ on the left, get

\begin{tiny}

$$\left( \begin{array}{ccccc}
1-y & 0 &\cdots & x_1-1&0\\
0 & 1-y &\cdots &x_2-1&0 \\
\vdots & \vdots &\ddots&\vdots& \vdots\\
0 & 0 &\cdots &1& 0\\
\end{array} \right)$$\end{tiny}

Add the last row times $1-x_i$ (on the left) to the $i$-th row, get
\begin{tiny}
$$\left( \begin{array}{ccccc}
1-y & 0 &\cdots &0&0\\
0 & 1-y &\cdots &0&0 \\
\vdots & \vdots &\ddots&\vdots& \vdots\\
0 & 0 &\cdots &1& 0\\
\end{array} \right)$$
\end{tiny}

To obtain a presentation matrix for $H_1(\U,u_0;\Kt)$, we choose the splitting that maps $t$ to $x_1$, then since $y$ has linking number $m$, $y$ is replaced by a polynomial $p(t)$ with degree $m$. So
$$H_1(\U,u_0;\Kt)\cong \Kt \oplus \frac{\Kt}{(p(t)-1)} \oplus \cdots \oplus \frac{\Kt}{(p(t)-1)} \oplus\frac{\Kt}{(1)}  $$ with $m-2$ copies of $\Kt/(p(t)-1) $. Thus $\delta_n=m(m-2).$

\begin{eg}\rm
$$G=\langle a, b \ | \ aba=bab \rangle$$
 is the trefoil knot group and also the fundamental group of the complement of a cuspidal cubic $C=\{x^2+y^3=0\}$. $\delta_0=2$ and $\delta_n=1$ for all $n\geq 1$.

\end{eg}

\begin{proof}
\textbf{Method 1}:  $x^2+y^3$ is a weight homogeneous polynomial with $w_1=3, w_2=2, d=6$.

By proposition 9.5, $$\mu (C,0)=\frac{(6-3)(6-2)}{3\cdot 2}=2.$$ Since $H_1(U)=\Z$, $\beta_1(\U)=1$. 
Thus $\delta_0=\mu (C,0)=2$ and $\delta_n=\mu (C,0)-1=1$ for all $n\geq 1$.

\bigskip
\textbf{Method 2}: (Use Fox Calculus)
Get a new set of generators by letting $x=a$, $y=ba^{-1}$. So, $$G=\langle x, y \ | \ xyx^2=yx^2yx\rangle =\langle x, y \ | \ xyxy^{-1}x^{-2}y^{-1}\rangle.$$
Since $
\frac{\partial R_1}{\partial x}=1+xy-xyxy^{-1}x^{-1}-xyxy^{-1}x^{-2}=1+xy-yx-y,$ and \newline
$\frac{\partial R_1}{\partial y}=x-xyxy^{-1}-xyxy^{-1}x^{-2}y^{-1}=x-yx^2-1,$ the presentation matrix is:
$$\left( \begin{array}{cc}
1+xy-yx-y & x-yx^2-1
\end{array} \right).$$

When $n=0$, since $G/G'=\langle x\rangle$, to obtain a presentation matrix for $H_1(\U,u_0;\Kt)$, choose splitting that maps $t \rightarrow x$, $1\rightarrow y$. Get
$$\left( \begin{array}{cc}
0& t-t^2-1
\end{array} \right)
\thicksim
\left( \begin{array}{cc}
t^2-t+1& 0
\end{array} \right)
$$
Therefore, $H_1(\U,u_0;\Kt)=\Kt\oplus \Kt/\langle t^2-t+1\rangle$ and we conclude that \linebreak
 $H_1(\U;\Kt)=\Kt/\langle t^2-t+1\rangle$. Thus $\delta_0=2$.

When $n\geq 1$, by theorem 3.6 of \cite{MR6} $G'/G''\cong \Z^2$ is generated by $y$ and $xyx^{-1}$. 
Thus $y-1\neq 0$ in $G'/G_r^{(n+1)}$, and $y-1$ is invertible in $\mathbb{K}_n$ for $n\geq 1$.
Now multiply the second column on the right by $1-y$. Add the first column times $1-x$ (on the right) to the second column, get
$$\left( \begin{array}{cc}
1+xy-yx-y & 0
\end{array} \right)$$
Choose splitting $t\rightarrow x$, $y\rightarrow y$, get
$$\left( \begin{array}{cc}
1+xyx^{-1}t-yt-y & 0
\end{array} \right)
\thicksim
\left( \begin{array}{cc}
(xyx^{-1}-y)t+(1-y)& 0
\end{array} \right)$$
Hence $H_1(\U;\Kt)=\Kt/\langle(xyx^{-1}-y)t+1-y \rangle$ and $\delta_n=1$ for $n\geq 1$.

\end{proof}

\begin{eg}\rm
$$G=\langle a, b \ | \ aba=bab \rangle \times \Z=\langle a, b, c \ | \ aba=bab, aca^{-1}c^{-1},bcb^{-1}c^{-1}\rangle$$

Let $C$ be the union of a cuspidal cubic and a generic line. By \cite{Oka}, $G=\pi_1(\C^2-C)$. Then
$\delta_n=0$ for $n\geq 1$.

\end{eg}

\begin{proof}
Change a new set of generators $x=a, y=ba^{-1}, z=ca^{-1}$, get
$$G=\langle x, y,z \ | \ xyxy^{-1}x^{-2}y^{-1},xzx^{-1}z^{-1},yzxy^{-1}x^{-1}z^{-1} \rangle.$$
Since $\frac{\partial R_2}{\partial x}=1-xzx^{-1}=1-z$, $\frac{\partial R_2}{\partial z}=x-xzx^{-1}z^{-1}=x-1$, 
 $\frac{\partial R_3}{\partial x}=yz-z$, $\frac{\partial R_3}{\partial y}=1-zx$, $\frac{\partial R_3}{\partial z}=y-1$,
the presentation matrix is 
\begin{tiny}
$$\left( \begin{array}{ccc}
1+xy-yx-y & x-yx^2-1 & 0\\
1-z & 0& x-1\\
yz-z&1-zx&y-1
\end{array} \right)$$
\end{tiny}

When $n=0$, $\Gamma_0=\Z^2$ generated by $x$, $z$. So $z\neq 1$ in $\mathbb{K}_n$ for all $n$, and $z-1$ is invertible. Choose splitting $t\rightarrow x$, $1\rightarrow y$ and $z\rightarrow z$, get
\begin{tiny}
$$\left( \begin{array}{ccc}
0 & t-t^2-1 & 0\\
1-z & 0& t-1\\
0&1-zt&0
\end{array} \right)$$
\end{tiny}
Multiply the first column on the right by $(1-z)^{-1}$, get
\begin{tiny}
$$\left( \begin{array}{ccc}
0 & t-t^2-1 & 0\\
1 & 0& t-1\\
0&1-zt&0
\end{array} \right)$$
\end{tiny}
Add the first column times $1-t$ to the last column, get 
\begin{tiny}
$$\left( \begin{array}{ccc}
0 & t-t^2-1 & 0\\
1 & 0& 0\\
0&1-zt&0
\end{array} \right)
\thicksim
\left( \begin{array}{ccc}
 t-t^2-1 &0 & 0\\
 0& 1& 0\\
1-zt&0&0
\end{array} \right)$$
\end{tiny}
Multiply the first row on the left by $z^2$, then add last row times $-zt+(z-1)$ to the first row, get
\begin{tiny}
$$
\left( \begin{array}{ccc}
-z^2+z-1 &0 & 0\\
 0& 1& 0\\
1-zt&0&0
\end{array} \right)
\thicksim
\left( \begin{array}{ccc}
z^2-z+1 &0 & 0\\
 0& 1& 0\\
1-zt&0&0
\end{array} \right)$$
\end{tiny}
Since $z^2-z+1$ has three terms, it cannot be equal to zero, and therefore is a unit in $\Kt$. Multiply the first row by $(z^2-z+1)^{-1}$, then add the first row times $zt-1$ to the last row, get 
\begin{tiny}
$$\left( \begin{array}{ccc}
1 &0 & 0\\
 0& 1& 0\\
0&0&0
\end{array} \right)$$
\end{tiny}
So $\delta_0=0$.

When $n\geq 1$, now $1-y$ and $1-z$ are both invertible in $\mathbb{K}_n$, multiply the second column by $1-y$, get
\begin{tiny}
$$\left( \begin{array}{ccc}
1+xy-yx-y & (x-yx^2-1)(1-y) & 0\\
1-z & 0& x-1\\
(y-1)z&(1-zx)(1-y)&y-1
\end{array} \right)$$
\end{tiny}
Add the first column times $1-x$ and the last column times $1-z$ to the second column, get
\begin{tiny}
$$\left( \begin{array}{ccc}
1+xy-yx-y & 0 & 0\\
1-z & 0& x-1\\
(y-1)z&0&y-1
\end{array} \right)
\thicksim
\left( \begin{array}{ccc}
1+xy-yx-y & 0 & 0\\
1-z & x-1&0 \\
(y-1)z&y-1&0
\end{array} \right)
$$
\end{tiny}
Since $y-1$ is invertible, multiply the last row (on the left) by $(y-1)^{-1}$, get
\begin{tiny}
$$\left( \begin{array}{ccc}
1+xy-yx-y & 0 & 0\\
1-z & x-1&0 \\
z&1&0
\end{array} \right)
$$
\end{tiny}
Add the second column times $-z$ to the first column, get
\begin{tiny}
$$\left( \begin{array}{ccc}
1+xy-yx-y & 0 & 0\\
1-xz & x-1&0 \\
0&1&0
\end{array} \right)
$$
\end{tiny}
Add the last row times $1-x$ to the second row, 
\begin{tiny}

$$\left( \begin{array}{ccc}
1+xy-yx-y & 0 & 0\\
1-xz & 0&0 \\
0&1&0
\end{array} \right)
$$\end{tiny}
Choose a splitting $t\rightarrow x$,
\begin{tiny}

$$\left( \begin{array}{ccc}
(xyx^{-1}-y)t+1-y & 0 & 0\\
1-zt & 0&0 \\
0&1&0
\end{array} \right)
$$\end{tiny}
Add the second row times $-(xyx^{-1}-y)z^{-1}$ to the first row,
\begin{tiny}

$$\left( \begin{array}{ccc}
yz^{-1}-xyx^{-1}z^{-1}+1-y & 0 & 0\\
1-zt & 0&0 \\
0&1&0
\end{array} \right)
$$\end{tiny}
$yz^{-1}-xyx^{-1}z^{-1}+1-y\neq 0  \Leftrightarrow  (y-xyx^{-1})\neq (y-1)z$ which is true. So $yz^{-1}-xyx^{-1}z^{-1}+1-y$ is a unit. We get $\delta_n=0$ for $n\geq 1$.

\end{proof}

\begin{eg}\rm
$G_1=\langle a, b \ | \ aba=bab\rangle$, $G_2=\langle c,d \ | \ cdc=dcd\rangle$, $$G=G_1\times G_2=\langle a,b,c,d \ | \ abab^{-1}a^{-1}b^{-1},cdcd^{-1}c^{-1}d^{-1},aca^{-1}c^{-1},bcb^{-1}c^{-1},ada^{-1}d^{-1},bdb^{-1}d^{-1}\rangle$$
Notice that $G$ is the fundamental group of the complement of two cuspidal cubics intersecting transversely by \cite{Oka}.
$\delta_n=0$ for $n\geq 1$.
\end{eg}

\begin{proof}

Change new generators $x=a, y=ba^{-1}, u=ca^{-1}, v=da^{-1}$, then $a=x,b=yx,$ $c=ux,d=vx$. Get
$G=\langle x,y,u,v \ | \ xyxy^{-1}x^{-2}y^{-1},uxvxuv^{-1}x^{-1}u^{-1}x^{-1}v^{-1},$

$xux^{-1}u^{-1},yuxy^{-1}x^{-1}u^{-1},xvx^{-1}v^{-1},yvxy^{-1}x^{-1}v^{-1}\rangle.$

$\frac{\partial R_2}{\partial x}=u+uxv-vxu-v$, $\frac{\partial R_2}{\partial u}=1+uxvx-vx$, $\frac{\partial R_2}{\partial u}=ux-vxux-1$.

The presentation matrix is 
\begin{tiny}

$$\left( \begin{array}{cccc}
1+xy-yx-y & x-yx^2-1 & 0&0\\
u+uxv-vxu-v & 0& 1+uxvx-vx&ux-vxux-1\\
1-u&0&x-1&0\\
(y-1)u&1-ux&y-1&0\\
1-v&0&0&x-1\\
(y-1)v&1-vx&0&y-1
\end{array} \right)$$\end{tiny}

When $n=0$, $\Gamma_0=\Z^2$ generated by $x$, $u$. So $u\neq 1$ in $\mathbb{K}_n$ for all $n$, and $u-1$ is invertible. Choose splitting $t\rightarrow x$, $1\rightarrow y$, $u\rightarrow u$ and $u\rightarrow v$, get

\begin{tiny}

$$\left( \begin{array}{cccc}
0 & t-t^2-1 & 0&0\\
0 & 0& 1+u^2t^2-ut&ut-u^2t^2-1\\
1-u&0&t-1&0\\
0&1-ut&0&0\\
1-u&0&0&t-1\\
0&1-ut&0&0
\end{array} \right)$$\end{tiny}
Add the fourth row times $-1$ to the last row, get
\begin{tiny}

$$\left( \begin{array}{cccc}
0 & t-t^2-1 & 0&0\\
0 & 0& 1+u^2t^2-ut&ut-u^2t^2-1\\
1-u&0&t-1&0\\
0&1-ut&0&0\\
1-u&0&0&t-1\\
0&0&0&0
\end{array} \right)$$\end{tiny}
Since $1-u$ is invertible, multiply the third column by $1-u$, then add the first column times $1-t$ and last column times $1-u$ to the third column, multiply the first column by $(1-u)^{-1}$, get
\begin{tiny}

$$\left( \begin{array}{cccc}
0 & t-t^2-1 & 0&0\\
0 & 0& 0&ut-u^2t^2-1\\
1&0&0&0\\
0&1-ut&0&0\\
1&0&0&t-1\\
0&0&0&0
\end{array} \right)$$\end{tiny}
Subtract the third row from the fifth row, get
\begin{tiny}

$$\left( \begin{array}{cccc}
0 & t-t^2-1 & 0&0\\
0 & 0& 0&ut-u^2t^2-1\\
1&0&0&0\\
0&1-ut&0&0\\
0&0&0&t-1\\
0&0&0&0
\end{array} \right)
\thicksim
\left( \begin{array}{cccc}
 t-t^2-1 &0 & 0&0\\
  0&ut-u^2t^2-1&0 &0\\
0&0&1&0\\
1-ut&0&0&0\\
0&t-1&0&0\\
0&0&0&0
\end{array} \right)$$\end{tiny}
Multiply the first row (on the left) by $u^2$, add the fourth row times $-ut+u-1$ to the first row, add the fifth row times $u^2t+(u-1)u$ to the second row, get
\begin{tiny}

$$\left( \begin{array}{cccc}
 u-u^2-1 &0 & 0&0\\
  0&u-u^2-1&0 &0\\
0&0&1&0\\
1-ut&0&0&0\\
0&t-1&0&0\\
0&0&0&0
\end{array} \right)$$\end{tiny}
Since $u-u^2-1$ has three terms, it is invertible. Get
\begin{tiny}

$$\left( \begin{array}{cccc}
 1 &0 & 0&0\\
  0&1&0 &0\\
0&0&1&0
\end{array} \right)$$\end{tiny}
Thus $\delta_0=0.$

When $n\geq 1$, since $1-y$ is also invertible, multiply the second column by $1-y$, then add the first column times $1-x$, the third column times $1-u$, and the fourth column times $1-v$ to the second column, get 
\begin{tiny}

$$\left( \begin{array}{cccc}
1+xy-yx-y & 0 & 0&0\\
u+uxv-vxu-v & 0& 1+uxvx-vx&ux-vxux-1\\
1-u&0&x-1&0\\
(y-1)u&0&y-1&0\\
1-v&0&0&x-1\\
(y-1)v&0&0&y-1
\end{array} \right)$$\end{tiny}
Multiply the fourth and the sixth row by $(1-y)^{-1}$, get
\begin{tiny}

$$\left( \begin{array}{cccc}
1+xy-yx-y & 0 & 0&0\\
u+uxv-vxu-v & 0& 1+uxvx-vx&ux-vxux-1\\
1-u&0&x-1&0\\
u&0&1&0\\
1-v&0&0&x-1\\
v&0&0&1
\end{array} \right)
\thicksim
\left( \begin{array}{cccc}
1+xy-yx-y & 0 & 0&0\\
u+uxv-vxu-v & 1+uxvx-vx&ux-vxux-1& 0\\
1-u&x-1&0&0\\
u&1&0&0\\
1-v&0&x-1&0\\
v&0&1&0
\end{array} \right)$$\end{tiny}
Add the fourth row times $1-x$ to the third row. Add the sixth row times $1-x$ to the fifth row, get
\begin{tiny}

$$\left( \begin{array}{cccc}
1+xy-yx-y & 0 & 0&0\\
u+uxv-vxu-v & 1+uxvx-vx&ux-vxux-1& 0\\
1-xu&0&0&0\\
u&1&0&0\\
1-xv&0&0&0\\
v&0&1&0
\end{array} \right)$$\end{tiny}

Add the second column times $-u$ and the third column times $-v$ to the first column, get 
\begin{tiny}$$\left( \begin{array}{cccc}
1+xy-yx-y & 0 & 0&0\\
vxuxv-uxvxu & 1+uxvx-vx&ux-vxux-1& 0\\
1-xu&0&0&0\\
0&1&0&0\\
1-xv&0&0&0\\
0&0&1&0
\end{array} \right)$$\end{tiny}
Add the fourth row times $-(1+uxvx-vx)$ and the sixth row times $-(ux-vxux-1)$ to the second row, get
\begin{tiny}

$$\left( \begin{array}{cccc}
1+xy-yx-y & 0 & 0&0\\
vxuxv-uxvxu & 0&0& 0\\
1-xu&0&0&0\\
0&1&0&0\\
1-xv&0&0&0\\
0&0&1&0
\end{array} \right)
\thicksim
\left( \begin{array}{cccc}
1+xy-yx-y & 0 & 0&0\\
vxuxv-uxvxu & 0&0& 0\\
1-xu&0&0&0\\
1-xv&0&0&0\\
0&1&0&0\\

0&0&1&0
\end{array} \right)$$\end{tiny}
Choose a splitting $t\rightarrow x$, get
\begin{tiny}

$$\left( \begin{array}{cccc}
1+xyx^{-1}t-yt-y & 0 & 0&0\\
vxuxvx^{-2}t^2-uxvxux^{-2}t^2 & 0&0& 0\\
1-ut&0&0&0\\
1-vt&0&0&0\\
0&1&0&0\\
0&0&1&0
\end{array} \right)$$\end{tiny}
Looking at the first and the third row, if $yu^{-1}-xyx^{-1}u^{-1}+1-y$ is a unit, we are done.
Looking at the third and the fourth row, if $u^{-1}-v^{-1}$ is a unit, we are done.
In fact, $u^{-1}-v^{-1}=0 \Leftrightarrow v=u$, which is not true in $\Kt$, so $u^{-1}-v^{-1}$ is a unit. Get
\begin{tiny}
$$\left( \begin{array}{cccc}
1& 0 & 0&0\\
0&1&0&0\\
0&0&1&0
\end{array} \right)$$\end{tiny} and $\delta_n=0$ for $n\geq 1$.

\end{proof}

Inspired by the result of example 9.9 and 9.10, we have a following conjecture.
\begin{conj}
If $C$ is formed by two curves which intersect transversally, then the higher-order degrees of the complement $\delta_n=0$ for $n\geq 1$.

\end{conj}

\section{Higher-Order Degrees of A Group }

Given a group $G$ and a map $\phi:G\rightarrow \Z$, the first higher-order degree is defined (denote $\delta_n (G)=\delta_{n,1}(G)$). In fact, the first higher-order degree is group invariant. $$\delta_n(G)=\text{rk}_{\mathbb{K}_n}(G_r^{(n+1)}/[G_r^{(n+1)},G_r^{(n+1)}]\otimes_{\Z\bar{\Gamma}_n} \mathbb{K}_n)$$

Fox calculus can be applied to calculate $\delta_n(G)$.

\begin{eg}
$\delta_n(\Z_p\ast\Z_q)=0$ for all $n$.
\end{eg}
\begin{proof}

$G=\langle a,b \ | \ a^p=1=b^q \rangle.$
Since $a^p=1\in [G_r^{(n)},G_r^{(n)}]$, $a\in G_r^{(n+1)}$ for all $n$. Similarly, $b\in G_r^{(n+1)}$ for all $n$. Thus 
$\Gamma_n=G/G_r^{(n+1)}=0$ and $\bar{\Gamma}_n=0$. So $\Z\bar{\Gamma}_n=\Z$ and $\mathbb{K}_n=\Q$.

Then $G_r^{(n+1)}/[G_r^{(n+1)},G_r^{(n+1)}]=G/G'$ and $\delta_n(G)=\text{rk}_{\Q}((\Z_p\times \Z_q)\otimes_{\Z}\Q)=0$ for all $n$.
\end{proof}

From this proof, we get a more generalized result:
\begin{prop}
Given $G$, let $H=\{h\in G \ | \ h^k=1 \text{ for some }k\geq 0\}$ and $N=N_G(H)$ be the normalizer of $H$. Define $\widehat{G}:=G/N$ and let $f:G\rightarrow \widehat{G}$ be the quotient map. Then $$\delta_n(G)=\delta_n(\widehat{G}).$$ 
\end{prop}

For example, 
$\delta_n(G\times \Z_d)= \delta_n(G)=\delta_n(G\ast \Z_d)$.

\begin{proof}
Given $h\in H$, then $h^k=1\in [G_r^{(n)},G_r^{(n)}]$ for some $k$ and for all $n$. So $h\in G_r^{(n+1)}$ and $G/G_r^{(n+1)}=\Gamma_n=\widehat{\Gamma}_n=\widehat{G}/\widehat{G}_r^{(n+1)}$. Note that $\bar{\Gamma}_n$ is the kernel of the map $\Gamma_n\rightarrow \Z$, so $\bar{\Gamma}_n=\bar{\widehat{\Gamma}}_n$. To show $\delta_n(G)=\delta_n(\widehat{G}),$ it is enough to show $$G_r^{(n+1)}/[G_r^{(n+1)},G_r^{(n+1)}]\otimes_{\Z\bar{\Gamma}_n} \mathbb{K}_n\cong \widehat{G}_r^{(n+1)}/[\widehat{G}_r^{(n+1)},\widehat{G}_r^{(n+1)}]\otimes_{\Z\bar{\Gamma}_n} \mathbb{K}_n.$$
Let $e$ denote the identity element of $\bar{\Gamma}_n$, note that $G_r^{(n+1)}/[G_r^{(n+1)},G_r^{(n+1)}]$ is a multiplicative group and $\mathbb{K}_n$ is the ring of fraction of $\Z\bar{\Gamma}_n$.
So \begin{align*}
[h]\otimes 1&=[h]\otimes (k\cdot e )(k\cdot e)^{-1}\\
&=[ehe^{-1}]^k\otimes (k\cdot e)^{-1}\\
&=[h]^k\otimes (k\cdot e)^{-1}\\
&=[h^k]\otimes (k\cdot e)^{-1}\\
&=[1]\otimes (k\cdot e)^{-1},
\end{align*}
which mean all the torsion elements are trivial in the tensor product. Hence the above isomorphism is satisfied.

\end{proof}

\begin{eg}(\cite{Os}, Example 3.2)\rm

If $\bar{C}$ is Zariski's three-cuspidal quartic, then $G=\pi_1(\C^2-C)=\langle a, b \ | \ aba=bab, a^2=b^2 \rangle.$
Thus $G' \cong \Z/3\Z$. So $\delta_n(C)=0$, for all $n$. For all other quartics, the corresponding group of the affine complement is abelian, so the higher-order degrees vanish again.
\end{eg}

\begin{prop} Assume $G$ and $F$ are the fundamental group of some topological spaces with at least one non-torsion element, then
\begin{align*}
\delta_n(G\ast F)&=\infty,\\
\bar{\delta}_n(G\ast F)&=\bar{\delta}_n(G)+\bar{\delta}_n(F),\\
r_n(G\ast F)&=r_n(G)+r_n(F)+1.
\end{align*}

\end{prop}

\begin{proof} By Prop 10.2, we can now assume $G$ and $F$ have no torsion elements.

Case 1:  If $G=\Z$ or $F=\Z$, wlog, we assume $F=\Z$.  By fox calculus, we get the Jacobian matrix:
\begin{tiny}

\[  \left( \begin{array}{cccc}
\chi \left(\frac{\partial r_1}{\partial x_1}\right) &  ... & \chi \left(\frac{\partial r_1}{\partial x_k}\right) &0\\
\vdots &  \ddots & \vdots& \vdots\\
\chi \left(\frac{\partial r_l}{\partial x_1}\right) &  ... & \chi \left(\frac{\partial r_l}{\partial x_k}\right) &0
 \end{array} \right)
 \thicksim
 \left( \begin{array}{ccccc}
p_1(t) & ...& 0 &0& 0\\
\vdots &\ddots&\vdots &\vdots &\vdots\\
0& \cdots &p_{u}(t)&0 & 0
\end{array} \right)
\]\end{tiny}
\begin{align*}
 H_1(\U,u_0;\Kt)&\cong \Kt \oplus\Kt \oplus \frac{\Kt}{p_1(t)}\oplus\cdots \oplus \frac{\Kt}{p_u(t)},\\
  H_1(\U; \mathbb{K}_n[t^{\pm 1}])&\cong \Kt \oplus\frac{\Kt}{p_1(t)}\oplus\cdots \oplus \frac{\Kt}{p_u(t)}.
\end{align*}

So $ \delta_n(G \ast \Z)=\infty$, $\bar{\delta}_n(G\ast \Z)=\bar{\delta}_n(G)$, $r_n(G \ast \Z)=r_n(G)+1$. Notice that since $\Z$ is abelian, $ \delta_n(\Z)=\bar{\delta}_n(\Z)=r_n(\Z)=0.$
\bigskip

Case 2: If $G$ and $F$ both have more than one generators, then we get the Jacobian matrix:

\begin{tiny}

\[  \left( \begin{array}{cccccc}
\chi \left(\frac{\partial r_1^G}{\partial x_1^G}\right) &  ... & \chi \left(\frac{\partial r_1^G}{\partial x_k^G}\right) &0&0&0\\
\vdots &  \ddots & \vdots& \vdots& \vdots& \vdots\\
\chi \left(\frac{\partial r_i^G}{\partial x_1^G}\right) &  ... & \chi \left(\frac{\partial r_i^G}{\partial x_k^G}\right) &0&0&0\\
0&0&0&\chi \left(\frac{\partial r_1^F}{\partial x_1^F}\right) &  ... & \chi \left(\frac{\partial r_1^F}{\partial x_l^F}\right)\\
\vdots &  \vdots & \vdots& \vdots& \ddots& \vdots\\
0&0&0&\chi \left(\frac{\partial r_j^F}{\partial x_1^F}\right) &  ... & \chi \left(\frac{\partial r_j^F}{\partial x_l^F}\right) 
 \end{array} \right)
 \thicksim
 \left( \begin{array}{cccccccc}
p_1(t) & ...& 0 &0& 0&0&0&0\\
\vdots &\ddots&\vdots &\vdots &\vdots&\vdots&\vdots&\vdots\\
0& \cdots &p_{u}(t)&0 & 0& 0& 0& 0\\
0 & 0& 0& 0& q_1(t) & ...& 0 & 0\\
\vdots &\vdots&\vdots &\vdots &\vdots&\ddots&\vdots&\vdots\\
0 & 0& 0& 0&0 & ...& q_v(t)&0 
\end{array} \right)
\]

\end{tiny}

\begin{align*}
 H_1(\U,u_0;\Kt)&\cong \Kt \oplus\Kt \oplus \frac{\Kt}{p_1(t)}\oplus\cdots \oplus \frac{\Kt}{p_u(t)}\oplus \frac{\Kt}{q_1(t)}\oplus\cdots \oplus \frac{\Kt}{q_v(t)},\\
 H_1(\U; \mathbb{K}_n[t^{\pm 1}])&\cong \Kt \oplus\frac{\Kt}{p_1(t)}\oplus \cdots \oplus \frac{\Kt}{p_u(t)}\oplus\frac{\Kt}{q_1(t)}\oplus \cdots \oplus \frac{\Kt}{q_v(t)}
\end{align*}

So $\delta_n(G\ast F)=\infty$ and $\bar{\delta}_n(G\ast F)=\bar{\delta}_n(G)+\bar{\delta}_n(F)$ and $r_n(G\ast F)=r_n(G)+r_n(F)+1$.

\end{proof}

\begin{rmk}\rm
 If $G$ and $F$ are the fundamental groups of some topological spaces with at least one non-torsion element, $G\ast F$ can not be the fundamental group of an affine hypersurface (or curve) complement.

\end{rmk}

\begin{rmk}\rm
For some fixed $n\geq 0$, the following are equivalent:

1) $\delta_n(G)=0$.

2) $G_r^{(n+1)}/[G_r^{(n+1)},G_r^{(n+1)}]\otimes_{\Z\bar{\Gamma}_n} \mathbb{K}_n=0$.

3) $G_r^{(n+1)}/[G_r^{(n+1)},G_r^{(n+1)}]$ is a torsion $\Z \bar{\Gamma}_n$-module.

\end{rmk}

\end{document}